\documentclass[11pt]{amsart}
\usepackage{amsmath}
\usepackage{amssymb}
\usepackage{amsthm}
\usepackage[msc-links,abbrev]{amsrefs}
\usepackage[noabbrev,capitalize]{cleveref}

\def\sideremark#1{\ifvmode\leavevmode\fi\vadjust{\vbox to0pt{\vss
 \hbox to 0pt{\hskip\hsize\hskip1em
 \vbox{\hsize3cm\tiny\raggedright\pretolerance10000
 \noindent #1\hfill}\hss}\vbox to8pt{\vfil}\vss}}}

\newtheorem{theorem}{Theorem}[section]
\newtheorem{proposition}[theorem]{Proposition}

\newtheorem{corollary}[theorem]{Corollary}

 \theoremstyle{definition}
 \newtheorem{definition}[theorem]{Definition}

 \theoremstyle{remark}
 \newtheorem{remark}[theorem]{Remark}

\numberwithin{equation}{section}

\usepackage{comment}
\usepackage{xcolor}

\addtocontents{toc}{\setcounter{tocdepth}{1}} 

\begin{document}

\title[First Eigenvalue Estimates]{First Eigenvalue Estimates for Asymptotically Hyperbolic Manifolds and their Submanifolds}
\author{Samuel P\'erez-Ayala}
\address{370 Lancaster Ave \\Haverford College  \\ Haverford\\PA 19041\\ USA}
\email{sperezayal@haverford.edu}
\author{Aaron J. Tyrrell}
\address{18A Department of Mathematics and Statistics \\ Texas Tech University \\ Lubbock \\ TX 79409 \\ USA}
\email{aatyrrel@ttu.edu}

\subjclass[2020]{}
\begin{abstract}
 We derive a sharp upper bound for the first eigenvalue $\lambda_{1,p}$ of the $p$-Laplacian on asymptotically hyperbolic manifolds for $1<p<\infty$. We then prove that asymptotically CMC submanifolds within asymptotically hyperbolic manifolds are themselves asymptotically hyperbolic. As a corollary, we show that for any minimal conformally compact submanifold $Y^{k+1}$ within $\mathbb{H}^{n+1}(-1)$, $\lambda_{1,p}(Y)=\left(\frac{k}{p}\right)^{p}$. We then obtain lower bounds on $\lambda_{1,2}(Y)$ in the case where minimality is replaced with a bounded mean curvature assumption and where the ambient space is a general Poincar\'e-Einstein space whose conformal infinity is of non-negative Yamabe type. In the process, we introduce an invariant $\hat \beta^Y$ for each such submanifold, enabling us to generalize a result due to Cheung-Leung in \cite{CheungLeung-Fu2001Eefs}.
\end{abstract}

\maketitle

\section{Introduction}

Let $\overline{X}^{n+1}$ be a compact manifold with boundary and interior $X$. A complete Riemannian metric $g_+$ on $X$ is said to be conformally compact (CC) of order $C^{k,\alpha}$ if $\overline{g} := r^2g_+$ extends as a metric to $\overline{X}$ which is $C^{k,\alpha}(\overline{X})$, where $r$ is a defining function for $M := \partial \overline{X}.$ We will assume this extension is $C^{3,\alpha}(\overline{X}),$ however some of our results hold under weaker assumptions. By a defining function for the boundary $M$ we mean
\begin{align*} r &> 0 \hspace{2mm}\text{on $X,$} \\
  r &= 0 \hspace{2mm}\text{on $M,$}\\ dr &\neq 0 \hspace{2mm}\text{on $M.$}
 \end{align*} 
If $\rho$ is another defining function for $M$, then the induced boundary metrics $(r^2g_+)|_{TM}$ and $(\rho^2g_+)|_{TM}$ are conformal to each other, giving rise to an invariant $[g_+]_{\infty}:=[\overline{g}|_{TM}]$ of $g_+$ known as the conformal infinity of $(X^{n+1},g_+)$. Here  $[\overline{g}|_{TM}]$ denotes the conformal class of $\overline{g}|_{TM}$. As observed by Mazzeo in \cite{mazzeo1986hodge}, conformally transforming the Riemann tensor of $g_+=:g$ yields
\begin{equation}\label{TransformationRiemannTensor}
 R_{ijkl}^{g_+}=-|dr|^2_{\overline{g}}(g_{ik}g_{jl}-g_{il}g_{jk})+O_{ijkl}(r^{-3}), 
 \end{equation}
 where $O_{ijkl}(r^{-3})$ is the component function of a covariant 4-tensor which is $O(r^{-3})$ as $r\to 0$. Notice that since $(g_+)^{ij}=r^2\overline{g}^{ij}$, this implies that the tensor corresponding to those components vanishes with respect to the $g_+$- norm as $r\to 0^+$. Moreover, (\ref{TransformationRiemannTensor}) also gives us that the sectional curvatures of $(X^{n+1},g_+)$ approach $-|dr|^2_{\overline{g}}$ at the boundary $M$. Hence, $ |dr|_{r^2g_+}^2$ restricted to $M$ is another invariant of the metric $g_+$ (see Section 2 in \cite{Graham}).
 
Asymptotically hyperbolic (AH) manifolds are conformally compact manifolds that have asymptotic sectional curvatures equal to a constant $-\kappa^2$, where $\kappa >0$; unless otherwise stated, we will assume this constant to be $-1$. We denote them by AH$(-\kappa^2)$, or just AH if $-\kappa^2=-1$. Note that, based on (\ref{TransformationRiemannTensor}), this constant $-\kappa^2$ equals $-|dr|^2_{\overline{g}}|_M$.

A special defining function for $M$ is a defining function $r$ such that  $|dr|_{\overline{g}}^2\equiv 1$ in a neighborhood of $M$ rather than just on $M$ itself. It is a lemma in Graham-Lee's work \cite{GrahamLee} (see also Lemma 2.1 in \cite{Graham}) that for every metric $g_{\infty}$ in the conformal infinity of an asymptotically hyperbolic manifold $(X^{n+1},g_+)$, there exists a special defining function $r$ such that $(r^2g_+)|_{TM}=g_{\infty}.$ A conformally compact manifold that also satisfies the Einstein condition $\mathrm{Ric}(g_+)=-ng_+$ is called a Poincar\'e-Einstein (PE) manifold 
or conformally compact Einstein manifold. For such a manifold, it can be verified by contracting \eqref{TransformationRiemannTensor} that $\left(|dr|_{\overline{g}}^2\right)|_M=1$, thus implying that Poincar\'e-Einstein manifolds are special cases of asymptotically hyperbolic manifolds. In summary,
 \[
 \{\text{PE}\}\subset \{\text{AH}\}\subset \{\text{CC}\}.
 \]

The model case of a conformally compact Riemannian manifold is the hyperbolic space $\mathbb{H}^{n+1}(-1)$. We use the Poincar\'e ball model:
\begin{equation}\label{HyperbolicBall}
    (\mathbb{H}^{n+1},g_H) := \left(\mathbb{B}^{n+1}, \left(\frac{2}{1 - |y|^2}\right)^2\cdot\sum_{i=1}^{n+1}(dy_i)^2\right),
\end{equation}
where $\mathbb{B}^{n+1}\subseteq \mathbb{R}^{n+1}$ denotes the standard open unit ball in Euclidean space. The functions $r_1,r_2:\mathbb{B}^{n+1}\rightarrow \mathbb{R}^{n+1}$, defined as $r_1(y) = 1 - |y|$ and $r_2(y) = \frac{1-|y|}{1+|y|}$, where $|\cdot|$ denotes the standard euclidean norm, are examples of defining functions for $\partial \mathbb{B}^{n+1} = \mathbb{S}^n$ with both $r_1^2g_H$ and $r_2^2g_H$ extending to metrics on $\overline{ \mathbb{B}}^{n+1}$. As discussed in \cite{Graham}, $r_2$ is an example of a special defining function. Moreover, $[g_H]_\infty = [(r_2^2g_H)|_{T\mathbb{S}^n}] = [(\text{standard metric on $\mathbb{S}^n$})]$ is the conformal infinity of $(\mathbb{H}^{n+1},g_H)$. Finally, $g_H$ has constant sectional curvatures equal to $-1$, and so $\mathrm{Ric}(g_H)=-ng_H$, which means that $(\mathbb{H}^{n+1}(-1),g_H)$ is also an example of a Poincar\'e-Einstein manifold.

We now proceed to introduce our main results. We begin with a discussion of upper bounds for Dirichlet eigenvalues on asymptotically hyperbolic manifolds, followed by an examination of general eigenvalue estimates on their submanifolds. All submanifolds $\iota: \overline{Y} \to \overline{X}$ are assumed to be immersed and we will often identify $\overline{Y}$ with $\iota(\overline{Y})$ when there is no ambiguity.

\subsection{Eigenvalue Estimates in Complete, non-Compact Manifolds}

Let $(X^{n+1},g_+)$ be a complete and non-compact manifold. For any bounded domain $\Omega$ with smooth boundary, we consider the Dirichlet eigenvalue problem
\begin{equation}\label{p-Equation}
    \begin{cases}
        \Delta_p f + \lambda |f|^{p-2}f = 0 & \quad\text{ in }\Omega,\\
        f = 0 & \quad\text{ on }\partial\Omega,
    \end{cases}
\end{equation}
where $\Delta_p f = \text{div}(|\nabla_{g_+}f|^{p-2}\nabla_{g_+}f)$ is the so-called $p$-Laplace operator and $1<p<\infty$. Notice that $\Delta_2=\Delta_{g_+}$ is linear as it is just the standard Laplace-Beltrami operator. Any $\lambda\in\mathbb{R}$ for which a nontrivial solution $f$ of (\ref{p-Equation}) exists is called a Dirichlet eigenvalue of $\Delta_p$ or simply a $p$-Dirichlet eigenvalue. Solutions to (\ref{p-Equation}) are not smooth, in general, and thus by ``solutions" we always mean a solution in the sense of distributions; see \cite{Lima} for a precise formulation. Our focus is on the first $p$-Dirichlet eigenvalue, whose variational characterization is given by
\begin{equation}\label{p-Eigenvalue}
    \lambda_{1,p}(\Omega) = \inf_{f\not = 0} \frac{\displaystyle\int_\Omega |\nabla_{g_+}f|^p\;dv_{g_+}}{\displaystyle\int_\Omega |f|^p\;dv_{g_+}}
\end{equation}
and where the infimum is taken over the space $W^{1,p}_o(\Omega)$ - the completion of $C^\infty_o(\Omega)$ under the Sobolev norm $\|\cdot\|_{L^{1,p}(g_+)}$. The reader can consult \cites{LEAn2006Epft,LindqvistPeter}, and references therein, for the definition of the first $p$-Laplace Dirichlet eigenvalue (\ref{p-Eigenvalue}) and basic properties of the associated eigenfunction. 

Due to domain monotonicity (Lemma 1.1 in \cite{DuMao}), it makes sense to define the first $p$-Dirichlet eigenvalue of $(X^{n+1},g_+)$ as follows:
\begin{definition}[First $p$-Dirichlet Eigenvalue]
    \begin{equation}
        \lambda_{1,p}(X) := \inf_{\Omega} \lambda_{1,p}(\Omega),
    \end{equation}
where the infimum is being taken over all bounded domains of $X$ with smooth boundary. 
\end{definition}
We emphasize that the definition of $\lambda_{1,p}(X)$ also depends on $g_+$, and so a better notation would be $\lambda_{1,p}(X,g_+)$. However, since $g_+$ will be fixed throughout the paper, we omit any reference to it and simply write $\lambda_{1,p}(X)$. In the case where $X= \mathbb{H}^{n+1}$, we will denote $g_+ = g_H$. When dealing with immersed submanifolds $\iota \colon Y^{k+1}\to X^{n+1}$, $\lambda_{1,p}(Y)$ is meant to be with respect to the induced metric $h_+ = {\iota}^*g_+$. 

The first $p$-Dirichlet eigenvalue is an invariant of $(X^{n+1},g_+)$ which is often difficult to compute. A lot of attention has been given to the case when $p=2$, that is, to the first Dirichlet eigenvalue of the Laplace-Beltrami operator, when $(X^{n+1},g_+)$ is assumed to be asymptotically hyperbolic. In \cite{MazzeoRafe1991UCaI}, Mazzeo showed that the essential spectrum of $\Delta_2 = \Delta_{g_+}$ consists of the ray $[\left(\frac{n}{2}\right)^2,\infty)$. As a consequence, he obtained the following upper bound: $\lambda_{1,2}(X) \le \left(\frac{n}{2}\right)^2$. For a general $p\in (1,\infty)$, upper bounds are known but under global assumptions on the Ricci curvature. In fact, on complete, $(n+1)$-dimensional manifolds $(X^{n+1},g_+)$ with Ricci curvature bounded from below, $\mathrm{Ric}(g_+)\ge -ng_+$,  classical techiniques due to Cheng \cite{ChengShiu-Yuen1975Ecta} yield $\left(\frac{n}{p}\right)^p$ as an upper bound for $\lambda_{1,p}(X)$; see \cites{SungChiung-JueAnna2014Sgea,HijaziOussama2020TCCo} for discussion on lower bounds under the same conditons on the Ricci curvature. This upper bound is sharp since $\lambda_{1,p}(\mathbb{H}^{n+1}(-1)) = \left(\frac{n}{p}\right)^p$, which will become clear in the discussion that follows. 

On an AH manifold $(X^{n+1},g_+)$, we argue that upper bounds on $\lambda_{1,p}(X)$ should only be influenced by the asymptotic behavior of the manifold at infinity, and that no global information should be needed to derive sharp estimates. Indeed, if $\Omega$ is a bounded domain with a smooth boundary in a collar neighborhood of $M=\partial \overline{X}$, then $\lambda_{1,p}(X)\le \lambda_{1,p}(\Omega)$. On the other hand, the sectional curvatures are uniformly approaching $-1$ near the boundary, suggesting that the geometry of $(X^{n+1},g_+)$ should reflect some properties of $\mathbb{H}^{n+1}$ near $M$. In particular, we expect $\lambda_{1,p}(\Omega)$ to be close to $\left(\frac{n}{p}\right)^p = \lambda_{1,p}(\mathbb{H}^{n+1}(-1))$. In our first theorem, we prove that this intuition is indeed correct by generalizing Mazzeo's upper estimate for any $p\in (1,\infty)$. 
\begin{theorem}\label{p-UpperBound}
    Let $(X^{n+1},g_+)$ be an asymptotically hyperbolic manifold. Then $\lambda_{1,p}(X)\le \left(\frac{n}{p}\right)^p$.
\end{theorem}

A few remarks are in order. To our knowledge, this result is novel for $p\not = 2$. The methods and techniques are also novel, even for the case when $p=2$, involving only the construction of a family of test functions ``near infinity" to capture the behavior of the manifold near its boundary. Additionally, we would like to point out that a slight generalization of Theorem \ref{p-UpperBound} is available for  manifolds which are AH$(-\kappa^2)$. Indeed, if a manifold $(X^{n+1}, g_+)$ is AH$(-\kappa^2)$, then $(X^{n+1}, \kappa^{2}g_+)$ would be AH$(-1)$ thanks to the scaling properties of the Riemann tensor and (\ref{TransformationRiemannTensor}). Since $\lambda_{1,p}(X,\kappa^{2}g_+) = \kappa^{-p}\lambda_{1,p}(X^{n+1},g_+)$ due to the variational characterization (\ref{p-Eigenvalue}), we derive the following consequence to Theorem \ref{p-UpperBound}:
\begin{corollary}\label{p-UpperBoundGeneral}
    Let $(X^{n+1},g_+)$ be AH$(-\kappa^2)$. Then $\lambda_{1,p}(X)\le \left(\frac{n\kappa}{p}\right)^p$.
\end{corollary}
\begin{remark}
Recall that, for any defining function $r$, $\kappa=(|dr|_{r^2g_+})|_{M}.$ So, the upper bound can be re-written as \begin{equation}
\lambda_{1,p}(X)\le \left(\frac{n(|dr|_{r^2g_+})|_{M}}{p}\right)^p. \end{equation}
\end{remark}
Different bounds have been derived for $\lambda_{1,p}(X)$ under various geometric assumptions and through a variety of methods. In \cite{CarvalhoFranciscoG.2022Otft}, Carvalho-Cavalcante generalized a classical result due to McKean \cite{McKean} and showed that for simply connected manifolds with negative sectional curvature bounded above by $-\kappa^2$, the first $p$-Dirichlet eigenvalue is at least $\left(\frac{n\kappa}{p}\right)^p$ - this bound is sharp as it is precisely what you obtain in $\mathbb{H}^{n+1}(-\kappa^2)$. Therefore, together with Corollary \ref{p-UpperBoundGeneral}, we derive
\begin{corollary}\label{Sharp-pEstimate}
     Let $(X^{n+1},g_+)$ be simply connected and AH$(-\kappa^2)$ with sectional curvature bounded above by $-\kappa^2$. Then $\lambda_{1,p}(X) = \left(\frac{n\kappa}{p}\right)^p$. 
\end{corollary}

There is limited literature on upper bounds for $\lambda_{1,p}(X)$ when $X$ is asymptotically hyperbolic or under other geometric assumptions. In \cite{Lima}, Lima-Montenegro-Santos obtained an upper bound for what they call the essential $p$-first eigenvalue in terms of the exponential volume growth $O(X)$ of the manifold. In particular, they proved that for manifolds with infinite volume, the essential $p$-first eigenvalue is bounded above by $\left(\frac{O(X)}{p}\right)^p$ (see Theorem 1.4 in \cite{Lima}). Their results, as well as their techniques, are a generalization of the corresponding result by Brooks \cite{BrooksRobert1981Arbg} when $p=2$. Bounds of the same type as in Theorem \ref{p-UpperBound} have been found in \cite{LiZhi2020UBot}, but there they do not work with Dirichlet boundary conditions.

\subsection{Eigenvalue Estimates on Submanifolds of AH Spaces}\label{EESAH}

We proceed with a discussion of eigenvalue estimates on submanifolds of asymptotically hyperbolic spaces. Let us start with some consequences of Theorem \ref{p-UpperBound} for a generic $p\in(1,\infty)$.

In \cite{DuMao}, Du and Mao, using techniques developed by Cheung-Leung in \cite{CheungLeung-Fu2001Eefs}, showed that if $Y$ is a $(k+1)$-dimensional, complete and noncompact immersed submanifold in $\mathbb{H}^{n+1}(-1)$ with norm of the mean curvature satisfying $|H|\le \alpha<k$, then 
\begin{equation}\label{p-EigenvalueLowerBound}
0<\left(\frac{k-\alpha}{p}\right)^p \le \lambda_{1,p}(Y) .
\end{equation}
For complete, $(k+1)$-dimensional minimal submanifolds in $\mathbb{H}^{n+1}(-1)$, the lower bound becomes $\left(\frac{k}{p}\right)^p$, which is sharp as it is the value on totally geodesic planes. We demonstrate that the upper bound is also sharp on a broader class of submanifolds of AH spaces that we call asymptotically minimal submanifolds. More generally, we define:
\begin{definition}\label{A-CMC}
    Assume $(X^{n+1},g_+)$ is a conformally compact manifold. Let $\overline{Y}^{k+1}$ be a compact manifold with boundary and interior $Y$, and let $\iota \colon \overline{Y}\to \overline{X}^{n+1}$ be a $C^{2,\alpha}$ immersion. If $(Y,\iota^*g_+)$ is conformally compact with the property that $\iota|_{\partial \overline{Y}}\colon\partial \overline{Y}\to M$ is an immersion such that $\iota(\overline{Y})$ meets $M$ transversely, then we say that $(Y^{k+1},\iota^*g_+)$ is a conformally compact submanifold of $(X^{n+1},g_+)$. Furthermore, if its mean curvature vector $H$ satisfies $g_+(H,H)=C^2+O(r)$ for some $C\in [0,\infty)$ and where $r$ is any defining function $M$, then we say that $(Y,\iota^*g_+)$ is asymptotically CMC with asymptotic mean curvature equal to $C$. In the case where $C=0$, we say that $Y$ is asymptotically minimal. 
\end{definition}
\begin{remark}
Given any pair of defining functions $r$ and $\rho$ for $M,$ $r=O(\rho).$ It follows that the asymptotic mean curvature as described in the previous definition is well-defined. 
\end{remark}
Conformally compact submanifolds are complete and non-compact. Some examples of conformally compact submanifolds in hyperbolic space are the totally geodesic $k$-cells and the spherical catenoids; a non-example is given by the horospheres. Horospheres are a non-example since although they are complete and non-compact, they do not have a boundary under compactification. Another class of conformally compact submanifolds are the type of submanifolds introduced by Graham and Witten in \cite{Graham-Witten}. These submanifolds are also called polyhomogeneously immersed submanifolds; for more information on these, the reader may consult \cite{marx2021variations} and \cite{Case}. Finally, some submanifolds arising from the asymptotic Plateau problem would also be examples of asymptotically minimal submanifolds (see \cite{coskunuzer2009asymptotic} for a survey on the asymptotic Plateau problem). 

We are now in a position to state our first result related to this class of submanifolds. We prove that any asymptotically CMC submanifold with asymptotic mean curvature equal to $C$ will be AH$\left(-\left(1-\frac{C^2}{(k+1)^2}\right)\right)$ itself, assuming that it lives within an AH space, thus allowing us to apply Corollary \ref{p-UpperBoundGeneral}.
\begin{proposition}\label{ACM}
    Let $Y^{k+1}$ be an asymptotically CMC submanifold of an asymptotically hyperbolic space $(X^{n+1},g_+)$ with asymptotic mean curvature $C.$ Then $Y$ is asymptotically hyperbolic with asymptotic sectional curvatures equal to $-\left(1-\frac{C^2}{(k+1)^2}\right)$. In particular, if $Y$ is asymptotically minimal, then $Y$ is asymptotically hyperbolic.
\end{proposition}

As a consequence of Proposition \ref{ACM} and Corollary \ref{p-UpperBoundGeneral}, we obtain:
\begin{corollary}\label{cor3}
    Let the setup be the same as that in \cref{ACM}. Then 
    \begin{equation}\label{ACM-eigenvalue}
    \lambda_{1,p}(Y^{k+1})\leq \left(\frac{k}{p}\right)^p\left(1-\frac{C^2}{(k+1)^2}\right)^{\frac{p}{2}}.
    \end{equation}
    In particular, if $Y$ is asymptotically minimal, then the upper bound becomes $\left(\frac{k}{p}\right)^p$.
\end{corollary}
\begin{remark}
It will be shown in \cref{ACM-Sec} that in our context $C<k+1$ always holds.
\end{remark}
 Together with (\ref{p-EigenvalueLowerBound}) for the minimal case ($\alpha=0$), we conclude that the lower bound is attained for minimal conformally compact submanifolds of $\mathbb{H}^{n+1}(-1)$ for any $p\in(1,\infty)$. In other words, the bound $\left(\frac{k}{p}\right)^p$ is optimal for these submanifolds. We summarize our discussion in the next corollary:

\begin{corollary}\label{DuMao-answer}
    Let $Y^{k+1}$ be a minimal conformally compact submanifold of $\mathbb{H}^{n+1}(-1).$ Then $\lambda_{1,p}(Y) = \left(\frac{k}{p}\right)^p$.
\end{corollary}

The spherical catenoids $M_{-1,1}^n\subset \mathbb{H}^{n+1}(-1)$, introduced by Do Carmo-Dajczer in \cite{DoCarmo-Dj}, all have the same first eigenvalue. This follows since they are all minimal conformally compact submanifolds of hyperbolic space. One should note that they are not all isometric to each other (if this were the case, then it would not be interesting that they have the same first eigenvalue). In particular, it can be shown that there is a $1$-parameter family of them such that any spherical catenoid is mapped to a member of this family by an isometry of hyperbolic space and such that no two members of this family are isometric. In an upcoming paper by the second author, these statements are shown to be true along with other results.

We turn our attention to acquiring bounds on the first Dirichlet eigenvalue ($p=2$) of submanifolds of asymptotically hyperbolic spaces. First, let us recall some previous work on eigenvalue estimates for submanifolds of hyperbolic space. As previously mentioned, for complete, simply connected, $(n+1)$-dimensional manifolds $X$ with sectional curvature bounded above by $-\kappa^2$, $\kappa\not=0$, we have 
\begin{equation}\label{McKeanEstimate}
 \left(\frac{n\kappa}{2}\right)^2 \le \lambda_{1,2}(X);
\end{equation}
see \cites{McKean, CarvalhoFranciscoG.2022Otft}. In \cite{CheungLeung-Fu2001Eefs}, Cheung-Leung extended inequality (\ref{McKeanEstimate}) to complete non-compact submanifolds $Y^{k+1}$ of hyperbolic space $\mathbb{H}^{n+1}(-1)$ with sectional curvature satisfying $|H|\le \alpha<k$ by showing 
\begin{equation}\label{CL-estimate}
 \left(\frac{k-\alpha}{2}\right)^2 \le \lambda_{1,2}(Y).
\end{equation}
Their result follows from the following Poincar\'e-type inequality (Proposition 1 in \cite{CheungLeung-Fu2001Eefs}): for all $f\in C^\infty_0(Y)$, we have
\begin{equation}
    \left(\frac{k-\alpha}{2}\right)^2\int_Y f^2\le \int_Y |\nabla f|^2.
\end{equation}
A key element in their work is the following. Let $p\in \mathbb{H}^{n+1}(-1)\setminus Y$, and set $x(y) := d_{g_H}(y,p)$, where $g_H$ is the hyperbolic metric introduced in (\ref{HyperbolicBall}).
Then $u:\mathbb{H}^{n+1}(-1)\rightarrow \mathbb{R}$, $u(y) := \cosh(x(y))$, satisfies 
\begin{equation}\label{EigenvalueEq}
\Delta_{g_H}u = (n+1)u
\end{equation}
on $\mathbb{H}^{n+1}(-1)$. In fact, $u$ is an example of what we call a Lee-eigenfunction: these are smooth, positive functions solving (\ref{EigenvalueEq}) and satisfying certain asymptotics at the boundary; see Definition \ref{Lee-Eigenfunction} for precise details. By noticing this, we are able to give a simpler argument in the spirit of Lee's \cite{LeeJohnM.1995Tsoa} that also generalizes their result to any such submanifold living in a general Poincar\'e-Einstein manifold $(X^{n+1},g_+)$ with conformal infinity having non-negative Yamabe invariant.

\begin{theorem}\label{lower}
Let $\iota\colon Y^{k+1}\to X^{n+1}$ be a complete and noncompact immersion into a Poincar\'e-Einstein space $(X^{n+1},g_+)$ whose conformal infinity $(M,[g_+]_{\infty})$ has non-negative Yamabe invariant, and let $u$ be a Lee-eigenfunction. Denote by $b(u):= \nabla_{g_+}^2u - \frac{1}{n+1}(\Delta_{g_+}u)g_+ = \nabla_{g_+}^2u - ug_+$ the trace-free hessian of $u$. If the mean curvature vector of $Y$ has  norm satisfying $|H|\leq \alpha$ for some constant $\alpha,$ and if $\beta^Y(u):= \underset{Y}{\sup} \left(u^{-1}\cdot tr\left(b(u)|_{(TY^{\perp}){}^2}\right)\right)$ satisfies $\alpha+\beta^Y(u)<k$, then it follows that 
\begin{equation}\label{lower-estimate1}
0<\left(\frac{k-\alpha-\beta^Y(u)}{2}\right)^2\leq\lambda_{1,2}(Y).
\end{equation} Note: $0\leq \beta^Y(u)$ since $u= O(r^{-1})$ and $tr\big(b(u)|_{(TY^{\perp}){}^2}\big)= O(1),$ as $r\to 0$.
\end{theorem}

Let $\mathcal{L}(g_+)$ be the collection of all Lee-eigenfunctions (see Definition \ref{Lee-Eigenfunction}). As explained in \cref{LEE} (see Appendix), there is a one-to-one correspondence between $\mathcal{L}(g_+)$ and those metrics in the conformal infinity of $(X^{n+1},g_+)$ whose scalar curvature is nowhere negative on $M$. By taking the infimum of $\beta^Y(u)$ over $u\in\mathcal{L}(g_+)$, an improved version of (\ref{lower-estimate1}) is obtained:
\begin{corollary}\label{cor1}
    Let the setup be the same as that in Theorem \ref{lower}, and define $\hat{\beta}^Y:=\underset{u\in\mathcal{L}(g_+)}{\inf}\beta^Y(u)$. Then \begin{equation}\label{ImprovedLower}
    \left(\frac{k-\alpha-\hat{\beta}^Y}{2}\right)^2\leq \lambda_{1,2}(Y).
    \end{equation}
    Note that $\hat{\beta}^Y$ is an invariant of $\iota^*g_+$.
\end{corollary}

Estimate (\ref{ImprovedLower}) can be thought of as a generalization of Cheung-Leung's estimate (\ref{CL-estimate}) to submanifolds living in a general Poincar\'e-Einstein space $(X^{n+1},g_+)$ whose conformal infinity has non-negative Yamabe invariant. When the ambient space is $(\mathbb{H}^{n+1}(-1),g_H)$, we observe that $\hat\beta^Y=0$ for every $Y$ and recover Cheung-Leung's result:
\begin{corollary}\label{cor2}
Let $Y^{k+1}$ be a complete and noncompact submanifold in hyperbolic space $\mathbb{H}^{n+1}(-1)$. If the mean curvature vector of $Y$ has norm satisfying $|H|\leq \alpha< k$, then it follows that $ \left(\frac{k-\alpha}{2}\right)^2\leq\lambda_{1,2}(Y)$.\end{corollary}

As another consequence of Theorem \ref{lower}, we derive a stability result for complete and non-compact minimal hypersurfaces in Poincar\'e-Einstein spaces whose conformal infinity has non-negative Yamabe invariant. Indeed, recall that a minimal hypersurface $Y^n$ living in a Poincar\'e-Einstein space (i.e. $\mathrm{Ric}(g_+) = -ng_+$) is stable if 
\begin{equation}
    \int_Y(|\nabla_{h_+}f|^2 - (|B|^2 -n))f^2\;dv_{h_+}\ge 0
\end{equation}
for every compactly supported function $f$, where $B$ denotes the second fundamental form of $Y$. It is expected then that lower bounds on $\lambda_{1,2}(Y)$ together with upper bounds on $|B|^2$ will yield stability results; for instance, see section 5 in \cite{DoCarmo-Dj}. This generalizes a result of Seo in \cite{SeoKeom-Kyo2011SMHI} (see remark \ref{Seo}).

\begin{corollary}\label{Stability}
    Let $Y^n$ be a complete and non-compact minimal hypersurface in a Poincar\'e-Einstein space $(X^{n+1},g_+)$ where the conformal infinity of $g_+$ is of non-negative Yamabe type. If $|B|^2\leq\frac{(n-1 -\hat{\beta}^Y)^2}{4} + n$ at every point of $Y$ and $\hat \beta^Y\le n-1$, then $Y$ is stable. If  $(X^{n+1},g_+)=(\mathbb{H}^{n+1}(-1),g_H)$, then $|B|^2\le \frac{(n+1)^2}{4}$ everywhere implies stability.
\end{corollary}

We now provide a geometric interpretation of the invariant $\hat\beta^Y$ introduced in Corollary \ref{cor1}. On any AH space $(X^{n+1},g_+)$, and for any defining function $r$, a unique $u$ solving $\Delta_{g_+}u=(n+1)u$ and satisfying $u-r^{-1} = O(1)$ exists; see discussion in Section \ref{LowerBound}. Therefore, it makes sense to compactify the metric $g_+$ using $u^{-1}$, that is, we consider the compact space $(\overline{X}^{n+1},u^{-2}g_+)$. It was shown by Qing in \cite{QingJie2003Otrf} that this compactification satisfies nice properties, allowing the author to obtain an important rigidity result for Poincar\'e-Einstein spaces whose conformal infinity is the round sphere. Inspired by the same approach, a simple calculation leads to the following:

\begin{proposition}\label{GeomInt} Let $(X^{n+1}, g_+)$ be a Poincar\'e-Einstein space and let $u$ be as described above. Set $g_u = u^{-2}g_+$. If $E_{g_u}$ denotes the trace-free Ricci tensor of $g_u$ and $b(u) = \nabla_{g_+}^2u - ug_+$ is the trace-free hessian of $u$, then $E_{g_u} = (n-1)\frac{b(u)}{u}$. Moreover, $b(u) = 0$ if and only if $R_{g_u}$ is constant.
\end{proposition}
Recall that in the context of \cref{lower}, for a given $u\in\mathcal{L}(g_+),$ in order to obtain $\beta^Y(u)$, one first need to restrict $u^{-1}b(u) = (n-1)^{-1}E_{g_u}$ to the normal bundle of $Y$, take the trace with respect to $(g_+)|_{
TY^\perp}$, and then take the supremum over $Y$. The invariant $\hat \beta^Y$ is then $\hat \beta^Y = (n-1)^{-1}\underset{u\in \mathcal{L}(g_+)}{\inf} \text{tr}_{g_+}\left(E_{g_u}|_{TY^\perp}\right)$. An interesting problem is to investigate for which submanifolds $Y$ is $\hat\beta^Y=0$. For instance, we could ask if there are submanifolds for which $\hat\beta^Y>0$ and whether or not (\ref{ImprovedLower}) is optimal in such cases.

We close our introduction with a few remarks about the estimate (\ref{ACM-eigenvalue}) in Corollary \ref{cor3}, particularly in the case where $p=2$, while reviewing some previous work. In \cite{SeoKeom-Kyo2011SMHI}, the author works with complete non-compact minimal hypersurfaces $Y^n$ in $\mathbb{H}^{n+1}(-1)$ which are stable, that is, for which the second variation of the area functional with respect to compactly supported normal variations satisfy a sign condition (see condition ($1.2$) in \cite{SeoKeom-Kyo2011SMHI}). Under a finite assumption on the $L^2$-norm of the second fundamental form of such submanifolds, Seo managed to prove that 
\begin{equation}\label{Seo bound}
    \lambda_{1,2}(Y)\le n^2.
\end{equation}
The key element in their proof is the construction of a test function for the Rayleigh quotient in an arbitrary geodesic ball involving the norm of the second fundamental form - this is how the finite assumption and the stability condition come into play. 

In Fu-Tao's work \cite{FuTao}, they managed to improve Seo's upper bound (\ref{Seo bound}). In fact, under similar integrability and stability assumptions, they show that for a complete and non-compact hypersurface $Y^n$ in $\mathbb{H}^{n+1}(-1)$ it holds that
\begin{equation}
    \lambda_{1,2}(Y)\le \left(\frac{n-1}{2}\right)^2\left(1-\frac{|H|^2}{n^2}\right),
\end{equation}
where $H$ is the scalar mean curvature of $Y$. In particular, if it is minimal, then it gives the sharp upper bound. They also obtain the same upper bound for higher codimensional submanifolds $Y^{k+1}$ under the assumption that $H$ is parallel and that $\displaystyle\int_Y |\mathring{B}|^q\;dv_{h}<\infty$ for $q\ge k+1$, where $\mathring{B}$ is the traceless second fundamental form  of $Y$ and $h = \iota^*g_H$ is the induced metric from $(\mathbb{H}^{n+1},g_H)$; see Theorem 1 in \cite{FuTao}. This integrability assumption always holds: indeed, let $r$ be any defining function for $(\mathbb{H}^{n+1},g_H)$, and denote by $\bar h$ the metric on $Y$ induced from the compactified metric $\overline{g}_H = r^2g_H$. Then, denoting by $\mathring{\overline{B}}$ the traceless second fundamental form with respect to $\bar h$ and using conformal covariance, we deduce
\begin{equation}
\int_Y|\mathring{B}|^q\;dv_h = \int_Yr^q\left|\mathring{\overline{B}}\right|^q\;dv_{h} = \int_Y r^{q-(k+1)}\left|\mathring{\overline{B}}\right|^q \;dv_{\bar h}<\infty,
\end{equation}
where the inequality follows from the fact that $\bar h$ extends smoothly to the compact set $\overline{Y}$ and the integrand is continuous there.

Finally, we would like to reiterate that upper bounds for Dirichlet eigenvalues should only be influenced by the asymptotic behavior of the manifold at infinity; see the discussion prior to the statement of Theorem \ref{p-UpperBound}. In particular, no global integrability assumptions on the second fundamental form or stability assumptions should be necessary if the manifold satisfies certain asymptotics at infinity.  This is precisely the case when the submanifolds are asymptotically CMC and the ambient manifold is AH. 

\subsection{Organization of the paper} In section \ref{Preliminaries}, we explain the asymptotic properties of an AH manifold in terms of a special defining function $r$ for its boundary $M$. In particular, we describe how the metric can be written in normal form in a neighborhood of the boundary and the specific asymptotics of the volume form there. In section \ref{Proof-p-UpperBound}, we provide the proof of one of our main results, Theorem \ref{p-UpperBound}. The techniques herein are novel and self-contained. Finally, in section \ref{SubmanifoldsOfAH}, we furnish the proofs of results concerning eigenvalue estimates on submanifolds of asymptotically hyperbolic manifolds, as elucidated in section \ref{EESAH}. From this section, there are two aspects that we would like to highlight. In subsection \ref{ACM-Sec}, while proving Proposition \ref{ACM}, we noticed that the asymptotic norm of the mean curvature determines both the angles at which the submanifold meets the boundary at infinity and the asymptotic sectional curvature, which could be of independent interest. On the other hand, in section \ref{LowerBound}, we explain what we refer to as a Lee-eigenfunction. 

The reader will notice that we employ two type of test functions throughout our work: $u^{-s}$ and $r^{s'}$ for particular $s,s'\in \mathbb{R}$. For the upper bounds, we utilize $r^{s'}$ as we only need to know what happens at infinity. For lower bounds, we employ $u^{-s}$ since we require not only good asymptotics at infinity ($u-r^{-1} = O(1)$) but also interior information ($\Delta_{g_+}u = (n+1)u$ everywhere on $X$).

\subsection{Acknowledgements}
This work was initiated during  the Internatio-nal Doctoral Summer School In Conformal Geometry and Non-local Operators at the Instituto de Matem\'aticas de la Universidad de Granada (IMAG). We thank the organizers, Azahara DelaTorre Pedraza and Mar\'ia del Mar Gonz\'alez, for their kind invitation. The first author was supported by the NSF grant RTG-DMS-1502424 while in his previous institution.

\section{Preliminaries for Estimates on Asymptotically Hyperbolic Manifolds}\label{Preliminaries}

Recall that given a compact manifold $\overline{X}^{n+1}$ with boundary $M$ and interior $X$, a complete Riemannian metric $g_+$ on $X$ is said conformally compact if $\overline{g}:=r^2g_+$ extends to a metric on $\overline{X}$. Here $r$ is a defining function for $M$, that is, a positive function on $X$ which vanishes on $M$ and with non-vanishing gradient on $M$. Furthermore, if the asymptotic sectional curvatures approach $-1$ at the boundary, then we say that $(X^{n+1},g_+)$ is asymptotically hyperbolic. 
 
 A defining function $r$ determines for some $r_0>0$ an identification of $M\times [0,r_0)$ with a neighborhood of $M$ in $X$, as follows:  any $(p,t)\in M\times [0,r_0)$ corresponds to $\phi(p,t),$ where $\phi$ is the flow of $\nabla_{\overline{g}}r.$ If  $r$ is a special defining function, that is, if $r$ is such that $|\nabla_{g_+}r|^2 =r^2$ in a neighborhood of $M$, then $r(\phi(p,t))=t$. This means that we can think about the $t$ coordinate as just being $r$ and $\nabla
_{\overline{g}}r$ is orthogonal to the slices $M\times \{t\}.$ The metric $\overline{g}$ then takes the form \begin{equation}\overline{g}=g_r+dr^2\end{equation} where $g_r$ is a 1-parameter family of metrics on $M$. It can be shown that
\begin{equation}
g_r=g_{0}-2Lr+O(r^2),
\end{equation}
where $L$ is the second fundamental form of $M$ with respect to $\overline{g}$ (see \cite{Gursky-Graham}).
In $M\times [0,r_0)_r$, we can write the volume form of $g_+$ as
\begin{equation}\label{VolumeExp1}
   dv_{g_+}=r^{-n-1} \bigg(\frac{\mathrm{det}g_r}{\mathrm{det} g_0}\bigg)^{1\slash  2} dv_{g_0} dr.
\end{equation}
Then using Jacobi's formula 
\begin{equation}
    \frac{d}{dt}(\mathrm{det} A(t))= (\mathrm{det} A(t))*\mathrm{trace}(A^{-1}(t)A'(t)),
\end{equation}
we deduce
\begin{equation}\label{VolumeExp2}
\bigg(\frac{\mathrm{det}g_r}{\mathrm{det}g_0}\bigg)^{1\slash  2}=1+\nu^{(1)}r+O(r^2)
\end{equation}
where $\nu^{(1)}=-H^M$ and $H^M$ is the mean curvature of $M$ with respect to $\overline{g}$.  

\section{Upper Bounds for \texorpdfstring{$\lambda_{1,p}(X)$} - Proof of Theorem \ref{p-UpperBound} and of Corollary \ref{p-UpperBoundGeneral}}\label{Proof-p-UpperBound}

\subsection{Proof of Theorem \ref{p-UpperBound}}
Our strategy to prove Theorem \ref{p-UpperBound} is quite standard. Using the variational characterization of $p$-Dirichlet eigenvalues, given in (\ref{p-Eigenvalue}), we compute the Rayleigh quotient for a suitable test function that will yield the sharp bound. 

In what follows, we make our approach precise. Let $r_0$ as in Section \ref{Preliminaries}. For $\gamma\in (0,r_0)$, consider the bounded domains given by 
\begin{equation}
X_\gamma:= \{x\in X^{n+1}: r(x)\ge \gamma\}.
\end{equation}
Notice that the boundary of each $X_\gamma$ is merely the $\gamma$-level set of $r$, and so it is a smooth hypersurface thanks to our choice of $r_0$. By the domain monotonicity property of $p$-Dirichlet eigenvalues (Lemma 1.1 in \cite{LEAn2006Epft}), we have that $\lambda_{1,p}(X_\gamma)$ decreases monotonically to $\lambda_{1,p}(X)$ as $\gamma\to 0^+$. Moreover, 
\begin{equation}\label{VarChar2}
\lambda_{1,p}(X)\leq\lambda_{1,p}(X_{\gamma})= \underset{f\in W^{1,p}_0(X_{\gamma})\setminus\{0\}}{\text{inf}}\frac{\displaystyle\int_{X_{\gamma}}|\nabla_{g_+}|f||^p\;dv_{g_+}}{\displaystyle\int_{X_{\gamma}}|f|^p\;dv_{g_+}}.
\end{equation}

We want to use $r^{\frac{n}{p}}$ as a test function because a special defining function should capture the behavior of $X$ at infinity. The issue is that $r$ is positive everywhere in the interior of $\overline{X}$ and so it cannot be used as a test function. However, we can multiply it by a suitable cut-off $\phi$ in such a way that we capture the same behavior near $M$. This cut-off function is indexed by two positive parameters $\epsilon$ and $\delta$, and it is define as follows: for $\epsilon\in(0,r_0)$ and $\delta>0$ small enough, let $\phi_{\epsilon\delta}\colon \overline{X}\to \mathbb{R}$ be 
\begin{equation} 
\phi_{\epsilon\delta}(p):=\begin{cases} 
      1, & \quad\epsilon\leq r(p) \\
      \frac{1+\delta}{\epsilon}(r(p)-\frac{\epsilon\delta}{1+\delta}), & \frac{\epsilon\delta}{1+\delta}\leq r(p)\leq \epsilon \\
      0, & \quad 0\leq r(p)\leq   \frac{\epsilon\delta}{1+\delta}.
   \end{cases}
\end{equation}

We also need to introduce some big O notation that will make some of the computations clearer. We say that a function $w$ is $O(a,b) = O(a(\epsilon), b(\delta))$ if there exist positive constants $A$, $B$, $\epsilon_o$ and $\delta_o$ such that $|w(\epsilon,\delta)|\le A|a(\epsilon)|$ and $|w(\epsilon,\delta)|\le B|b(\delta)|$ for all $0\le \epsilon<\epsilon_o$ and for all $0\le \delta<\delta_o$. For instance, a function $w(\epsilon,\delta)$ being $O(\epsilon^2,1)$ means that, in particular, the function remains bounded as $\delta\to 0^+$ while for any fixed $\delta>0$  it is $O(\epsilon^2)$ as $\epsilon\to 0^+$. A function is $O(1,1)$ if it remains bounded as both $\epsilon\to 0^+$ and $\delta\to 0^+$. As in other sections of our work, we employ classical big O notation as well. For instance, a function $t(r)$ is $O(r)$ if there exist positive constants $R$ and $r_o$ such that $|t(r)|\le R\cdot r$ for all $0\le r\le r_o$.

\begin{proof}[Proof of Theorem \ref{p-UpperBound}]
Set
\begin{equation}  q_p({\epsilon,\delta},s):=\frac{\displaystyle\int_{X}|\nabla_{g_+}( r^s\phi_{\epsilon\delta})|^p\;dv_{g_+}}{\displaystyle\int_{X} r^{ps}\phi_{\epsilon\delta}^p\;dv_{g_+}} = \frac{N_p(\epsilon,\delta,s)}{D_p(\epsilon,\delta,s)},
\end{equation}
where $\frac{n-p}{p} < s < \frac{n}{p}$, and notice that thanks to (\ref{VarChar2}) we have
\begin{equation}\label{Eigenvalue-qfunction}
    \lambda_{1,p}(X)\le \lim_{\epsilon\to 0^+}\lim_{\delta\to 0^+}q_p(\epsilon,\delta,s).
\end{equation}
Also, combining (\ref{VolumeExp1}) and (\ref{VolumeExp2}), we get 
\begin{equation}\label{VolumeExp3}
dv_{g_+} = r^{-n-1}(1-H^Mr+O(r^2))\; dv_{g_0}dr = r^{-n-1}[1+O(r)]\; dv_{g_0}dr.
\end{equation}
The range in the values of $s$ will be explained as we proceed with the proof.

We begin by studying the denominator in $q_p(\epsilon,\delta,s)$:
\[
\begin{split}
D_p(\epsilon,\delta,s) =& \int_{\{r\ge r_0\}} r^{ps}\;dv_{g_+} + \int_\epsilon^{r_0} \int_M   r^{ps-n-1}[1+O(r)]\; dv_{g_0}dr \\ &\;+ \left(\frac{1+\delta}{\epsilon}\right)^p\int^\epsilon_{\frac{\epsilon\delta}{1+\delta}}\int_M \left(r - \frac{\epsilon\delta}{1+\delta}\right)^p\cdot r^{ps-n-1}[1+O(r)]\;dv_{g_0}dr \\ =&\; O(1,1) + \frac{\text{Vol}_{g_0}(M)}{ps-n}\cdot \left[ r^{ps-n}\right|_\epsilon^{r_o} + O(\epsilon^{ps-n+1},1) \\ &\;+ \left(\frac{1+\delta}{\epsilon}\right)^p\int^\epsilon_{\frac{\epsilon\delta}{1+\delta}} \int_M\left(r - \frac{\epsilon\delta}{1+\delta}\right)^p\cdot r^{ps-n-1}[1+O(r)]\;dv_{g_0}dr
\end{split}
\]
For $\frac{\epsilon \delta}{1+\delta}<r<\epsilon$, we can expand 
\[
    \left(r - \frac{\epsilon\delta}{1+\delta}\right)^p = r^p - \left(pr^{p-1}\epsilon \left(\frac{\delta}{1+\delta}\right) - \frac{p(p-1)}{2}r^{p-2}\epsilon^2\left(\frac{\delta}{1+\delta}\right)^2 +\cdots \right).
\]
Therefore, $D_p(\epsilon,\delta,s)$ equals
\[
\begin{split}
&\; O(1,1) - \frac{\text{Vol}_{g_0}(M)}{ps-n}\epsilon^{ps-n} + O(\epsilon^{ps-n+1},1) \\ &+ \left(\frac{1+\delta}{\epsilon}\right)^p\int_{\frac{\epsilon\delta}{1+\delta}}^\epsilon\int_M \Bigg{\{}r^{p+ps-n-1} - pr^{p+ps-n-2}\epsilon \left(\frac{\delta}{1+\delta}\right)+\cdots \Bigg{\}}\times\\&
\hspace{3.5in}[1+O(r)]\;dv_{g_0}dr \\ =&\; O(1,1) - \frac{\text{Vol}_{g_0}(M)}{ps-n}\epsilon^{ps-n} + O(\epsilon^{ps-n+1},1)+ O(\epsilon^{ps-n+1},\delta) \\ &+ \text{Vol}_{g_0}(M)\left(\frac{1+\delta}{\epsilon}\right)^p \cdot \Bigg{[}\frac{r^{p+ps-n}}{p+ps-n}\\&\hspace{2in} - p\epsilon\left(\frac{\delta}{1+\delta}\right)\frac{r^{p+ps-n-1}}{p+ps-n-1} +\cdots\Bigg{|}^\epsilon_{\frac{\epsilon\delta}{1+\delta}} \\=&\;O(1,1) - \frac{\text{Vol}_{g_0}(M)}{ps-n}\epsilon^{ps-n} + O(\epsilon^{ps-n+1},1)+ O(\epsilon^{ps-n},\delta^{p+ps-n}) \\ &+\text{Vol}_{g_0}(M)\left(\frac{1+\delta}{\epsilon}\right)^p \frac{1}{p+ps-n}\left(\epsilon^{p+ps-n} - \left(\frac{\epsilon\delta}{1+\delta}\right)^{p+ps-n}\right) \\ &-\text{Vol}_{g_0}(M)\left(\frac{1+\delta}{\epsilon}\right)^p \frac{p\epsilon}{p+ps-n-1}\left(\frac{\delta}{1+\delta}\right)\times\\&\hspace{2in}\left(\epsilon^{p+ps-n-1} - \left(\frac{\epsilon\delta}{1+\delta}\right)^{p+ps-n-1}\right) \\ =&\; O(1,1) - \frac{\text{Vol}_{g_0}(M)}{ps-n}\epsilon^{ps-n} + O(\epsilon^{ps-n+1},1) + \text{Vol}_{g_0}(M)\frac{(1+\delta)^p}{p+ps-n}\epsilon^{ps-n} \\&- \text{Vol}_{g_0}(M)\frac{\delta^{p+ps-n}}{p+ps-n}\left(\frac{\epsilon}{1+\delta}\right)^{ps-n}\\& - \text{Vol}_{g_0}(M)(1+\delta)^{p-1}\frac{p\epsilon^{ps-n}\delta}{p-ps-n-1}\\ &+\text{Vol}_{g_0}(M) \frac{p\delta^{p+ps-n}}{p+ps-n-1}\left(\frac{\epsilon}{1+\delta}\right)^{ps-n} + O(\epsilon^{ps-n},\delta^{p+ps-n}).
\end{split}
\]
That is, $D_p(\epsilon,\delta,s)$ equals
\begin{equation}\label{Denominator1}
\begin{split}
   O(1,1) &+ \text{Vol}_{g_0}(M)\epsilon^{ps-n}\left(\frac{(1+\delta)^p}{p+ps-n} - \frac{1}{ps-n}\right) + O(\epsilon^{ps-n+1},1) \\&+O(\epsilon^{ps-n},\delta^{p+ps-n}).
\end{split}
\end{equation}
Notice that some of the powers in $\delta$ are $p+ps-n$. Later we will take the limit as $\delta \to 0^+$, and so we need $p+ps-n>0$, that is, we need $s>\frac{n-p}{p}$. 

We turn our focus to the numerator in $q_p(\epsilon,\delta,s)$. Recall that $r$ is a special defining function, and thus $|\nabla_{g_+}r|^2 = r^2$ in a collar neighborhood of $M$. Then $N_p(\epsilon,\delta,s)$ equals
\[
\begin{split}
     &\;\int_{X}|\nabla_{g_+} (r^s\phi_{\epsilon\delta})|^pdv_{g_+} \\=& \int_{\{r\ge r_0\}}|\nabla_{g_+}(r^s)|^p\;dv_{g_+}+ \int_{X\setminus\{r\ge r_0\}} | r^s\nabla_{g_+} \phi_{\epsilon\delta} +s r^{s-1}\phi_{\epsilon\delta}\nabla_{g_+} r|^p\;dv_{g_+}\\ =&\; O(1,1)+\int_{X\setminus\{r\ge r_0\}} (r^{2s}|\nabla_{g_+} \phi_{\epsilon\delta}|^2 + s^2r^{2(s-1)}\phi_{\epsilon\delta}^2|\nabla_{g_+} r|^2\\& \hspace{2in}+ 2sr^{2s-1}\phi_{\epsilon\delta}\langle\nabla_{g_+}r,\nabla_{g_+}\phi_{\epsilon\delta}\rangle_{g_+})^{\frac{p}{2}}\;dv_{g_+} \\ =&\; O(1,1) +  \int_\epsilon^{r_o} \int_M s^p r^{p(s-1)}r^pr^{-n-1}[1+O(r)]\;dv_{g_o}dr\\ &+ \int_{\frac{\epsilon\delta}{1+\delta}}^\epsilon \int_M (r^{2s}|\nabla_{g_+} \phi_{\epsilon\delta}|^2 + s^2r^{2(s-1)}\phi_{\epsilon\delta}^2r^2\\&\hspace{1in} + 2sr^{2s-1}\phi_{\epsilon\delta}\langle\nabla_{g_+}r,\nabla_{g_+}\phi_{\epsilon\delta}\rangle_{g_+})^{\frac{p}{2}}r^{-n-1}[1+O(r)]\;dv_{g_0}dr \\=&\; O(1,1)  - \frac{\text{Vol}_{g_0}(M)}{ps-n}s^p\epsilon^{ps-n} +O(\epsilon^{ps-n+1},1)\\ &+ \int_{\frac{\epsilon\delta}{1+\delta}}^\epsilon \int_M (r^{2s}|\nabla_{g_+} \phi_{\epsilon\delta}|^2 + s^2r^{2(s-1)}\phi_{\epsilon\delta}^2r^2\\&\hspace{1in} + 2sr^{2s-1}\phi_{\epsilon\delta}\langle\nabla_{g_+}r,\nabla_{g_+}\phi_{\epsilon\delta}\rangle_{g_+})^{\frac{p}{2}}r^{-n-1}[1+O(r)]\;dv_{g_0}dr
\end{split}
\]
Using $\nabla_{g_+}\phi_{\epsilon\delta} = \frac{1+\delta}{\epsilon}\nabla_{g_+}r$ on $\frac{\epsilon\delta}{1+\delta}< r< \epsilon$, the previous expression simplifies to
\begin{equation}\label{Numerator1}
\begin{split}
&\; O(1,1) - \frac{\text{Vol}_{g_0}(M)}{ps-n}s^p\epsilon^{ps-n} +O(\epsilon^{ps-n+1},1) \\&+ \int_{\frac{\epsilon\delta}{1+\delta}}^\epsilon \int_M \Bigg{(}r^{2s+2}\left(\frac{1+\delta}{\epsilon}\right)^2+ s^2r^{2s}\left(\frac{1+\delta}{\epsilon}\right)^2\left(r - \frac{\epsilon\delta}{1+\delta}\right)^2\\&\hspace{1in}+2sr^{2s+1}\left(\frac{1+\delta}{\epsilon}\right)^2\left(r - \frac{\epsilon\delta}{1+\delta}\right)\Bigg{)}^{\frac{p}{2}}r^{-n-1}[1+O(r)]\;dv_{g_0}dr
\end{split}
\end{equation}
We now proceed to simplify the main factor in the integrand:
\[
\begin{split}
&r^{2s+2}\left(\frac{1+\delta}{\epsilon}\right)^2 + s^2r^{2s}\left(\frac{1+\delta}{\epsilon}\right)^2\left(r - \frac{\epsilon\delta}{1+\delta}\right)^2\\&\hspace{1.5in}+2sr^{2s+1}\left(\frac{1+\delta}{\epsilon}\right)^2\left(r - \frac{\epsilon\delta}{1+\delta}\right) \\ &= \left(\frac{1+\delta}{\epsilon}\right)^2\Bigg{\{}r^{2s+2}+s^2r^{2s}\left(r^2-\frac{2r\epsilon\delta}{1+\delta} + \left(\frac{\epsilon\delta}{1+\delta}\right)^2\right) \\&\hspace{2.5in}+ 2sr^{2s+1}\left(r-\frac{\epsilon\delta}{1+\delta}\right)\Bigg{\}}\\&= \left(\frac{1+\delta}{\epsilon}\right)^2\Bigg{\{}r^{2s+2}+s^2r^{2s+2} - \frac{2s^2r^{2s+1}\epsilon\delta}{1+\delta} + \frac{s^2r^{2s}\epsilon^2\delta^2}{(1+\delta)^2}\\&\hspace{2.5in} + 2sr^{2s+2} - \frac{2sr^{2s+1}\epsilon\delta}{1+\delta}\Bigg{\}} \\ &=\left(\frac{1+\delta}{\epsilon}\right)^2\left\{r^{2s+2}(s^2+2s+1) -\frac{2r^{2s+1}\epsilon\delta}{1+\delta}(s^2+s)  + \frac{s^2r^{2s}\epsilon^2\delta^2}{(1+\delta)^2}\right\} \\ &= \left(\frac{1+\delta}{\epsilon}\right)^2r^{2s+2}(s^2+2s+1) + \epsilon^{-1}r^{2s+1}\cdot O(1,\delta) + r^{2s}\cdot O(1,\delta^2)
\end{split}
\]

Recall that we will first take the limit $\delta\to 0^+$. Therefore, we rewrite the previous expression as 
\begin{equation}\label{Numerator4}
    \left(\frac{1+\delta}{\epsilon}\right)^2r^{2s+2}(s^2+2s+1) + r^{2s+1}O(\delta) + r^{2s}O(\delta^2).
\end{equation}
After raising to the power of $\frac{p}{2}$, we obtain
\[
    \begin{split}
        &\left(\frac{1+\delta}{\epsilon}\right)^pr^{ps+p}(s^2+2s+1)^{\frac{p}{2}} + \frac{p}{2}r^{(2s+2)(\frac{p}{2} - 1)}(r^{2s+1}O(\delta) + r^{2s}O(\delta^2))\\ &\hspace{.2in}+ \frac{1}{2}\cdot\frac{p}{2}\left(\frac{p}{2}-1\right)r^{(2s+2)(\frac{p}{2} - 2)}(r^{2s+1}O(\delta) + r^{2s}O(\delta^2))^2+\cdots \\ &=\left(\frac{1+\delta}{\epsilon}\right)^pr^{ps+p}(s^2+2s+1)^{\frac{p}{2}} + \frac{p}{2}r^{(2s+2)(\frac{p}{2} - 1)}(r^{2s+1}O(\delta) + r^{2s}O(\delta^2))\\  &\hspace{.2in} + \frac{1}{2}\cdot\frac{p}{2}\left(\frac{p}{2}-1\right)r^{(2s+2)(\frac{p}{2} - 2)}(r^{2s+1}O(\delta) + r^{2s}O(\delta^2))^2+\cdots \\ &= \left(\frac{1+\delta}{\epsilon}\right)^pr^{ps+p}(s^2+2s+1)^{\frac{p}{2}} + \frac{p}{2}(r^{ps+p-1}O(\delta)+ r^{ps+p-2}O(\delta^2))+\cdots,
    \end{split}
\]
and notice that, starting from the second term above, we have a Laurent expansion in both $r$ and $\delta$ such that the exponent of $r$ and that of $\delta$ in each term always adds up to $ps+p$. Moreover, the power of $\delta$ is positive. Now, multiplying by $r^{-n-1}[1+O(r)]$, gives
\[
    \begin{split}
        &\left(\frac{1+\delta}{\epsilon}\right)^pr^{ps+p-n-1}(s^2+2s+1)^{\frac{p}{2}} + \frac{p}{2}(r^{ps+p-n-2}O(\delta)+ r^{ps+p-n-3}O(\delta^2))\\ &\hspace{.1in}+F_1(r,\delta) + \epsilon^{-p}O(r^{ps+p-n})+ O(r^{ps+p-n-1},\delta) + O(r^{ps+p-n-2},\delta^2)\\ &\hspace{.1in}+ F_2(r,\delta),
    \end{split}
\]
where $F_1(r,\delta)$ is a Laurent expansion in both $r$ and $\delta$ such that in each term the exponents of $r$ and $\delta$ add up to $ps+p-n-1$, while $F_2(r,\delta)$ is the same but the powers add up to $ps+p-n$. Once again, the power of $\delta$ is positive in every term of the expansions $F_1$ and $F_2$. 

Recall that $ps+p-n>0.$ Therefore we have the following two cases to consider. First, let us assume that $ps+p-n-k\not=-1$ for all $k\in \mathbb{N}\setminus\{1\}$. Performing the double integral in (\ref{Numerator1}) gives 
\begin{equation}
    \begin{split}
    &\frac{\text{Vol}_{g_0}(M)}{ps+p-n}\left(1+\delta\right)^p(s^2+2s+1)^{\frac{p}{2}}\epsilon^{ps-n}\left(1 - \left(\frac{\delta}{1+\delta}\right)^{ps+p-n}\right)\\ &\hspace{.2in} + [r^{ps+p-n-1}\Big{|}^\epsilon_{\frac{\epsilon\delta}{1+\delta}} \cdot O(\delta) + [r^{ps+p-n-2}\Big{|}^\epsilon_{\frac{\epsilon\delta}{1+\delta}} \cdot O(\delta^2) + O(\delta^\alpha) \\ &\hspace{.2in} + O(\epsilon^{ps-n+1},1) + [r^{ps+p-n}\Big{|}^\epsilon_{\frac{\epsilon\delta}{1+\delta}}\cdot O(\delta) + [r^{ps+p-n-1}\Big{|}^\epsilon_{\frac{\epsilon\delta}{1+\delta}}\cdot O(\delta^2),
    \end{split}
\end{equation}
where $\alpha>0$. The terms coming from integrating $F_1$ and $F_2$ have both been absorbed into the term $O(\delta^\alpha)$. Taking the limit as $\delta\to 0^+$, we obtain
\begin{equation}\label{Numerator5}
    \frac{\text{Vol}_{g_0}(M)}{ps+p-n}(s^2+2s+1)^{\frac{p}{2}}\epsilon^{ps-n} + O(\epsilon^{ps-n+1},1).
\end{equation}
On the other hand, if $ps+p-n-k=-1$ for some $k\in\mathbb{N}\setminus \{1\}$, then log terms will appear somewhere after the first term. However, in those cases, the contribution will be of the form $\log(\delta)\cdot O(\delta^\alpha)$, where $\alpha>0$. Therefore, after taking $\delta\to 0^+$, we still get (\ref{Numerator5}).

Combine (\ref{Numerator1}) and (\ref{Numerator5}) to get
\begin{equation}\label{Numerator3}
\begin{split}
    \lim_{\delta\to 0^+} N_p(\epsilon,\delta,s) =\; O(1,1) &- \frac{\text{Vol}_{g_0}(M)}{ps-n} s^p\epsilon^{ps-n} + O(\epsilon^{ps-n+1},1) \\&+\frac{\text{Vol}_{g_0}(M)}{ps+p-n}(s^2+2s+1)^{\frac{p}{2}}\epsilon^{ps-n}.
\end{split}
\end{equation}
Putting (\ref{Denominator1}) and (\ref{Numerator3}) together gives
\[
\begin{split}
    \lim_{\delta\to 0^+}&q_p(\epsilon,\delta,s)\cdot\Bigg{\{}O(1,1) + \text{Vol}_{g_0}(M)\epsilon^{ps-n}\left(\frac{1}{p+ps-n} - \frac{1}{ps-n}\right) \\&\hspace{2in}+ O(\epsilon^{ps-n+1},1)\Bigg{\}} \\&= O(1,1) - \frac{\text{Vol}_{g_0}(M)}{ps-n} s^p\epsilon^{ps-n} + O(\epsilon^{ps-n+1},1) \\&\hspace{2in}+ \frac{\text{Vol}_{g_0}(M)}{ps+p-n}(s^2+2s+1)^{\frac{p}{2}}\epsilon^{ps-n}.
\end{split}
\]
Let us recall that $n - ps>0$. Multiply across by $\epsilon^{n-ps}$, let $\epsilon\to 0^+$ and divide by $\text{Vol}_{g_0}(M)$ to obtain
\begin{equation}
\lim_{\epsilon\to 0^+}\lim_{\delta\to 0^+}q_p(\epsilon,\delta,s)\cdot \left\{\frac{1}{p+ps-n} - \frac{1}{ps-n}\right\} = \frac{(s^2+2s+1)^{\frac{p}{2}}}{p+ps-n} - \frac{s^p}{ps-n}
\end{equation}
Going back to (\ref{Eigenvalue-qfunction}), we conclude that 
\begin{equation}
\lambda_{1,p}(X)\le s^p\left(s+1-\frac{n}{p}\right)+ (1+2s+s^2)^{\frac{p}{2}} \left(\frac{n}{p} - s\right)
\end{equation}
for all $\frac{n-p}{p}<s<\frac{n}{p}$. Letting $s\to \frac{n-p}{p}^+$ or $s\to \frac{n}{p}^-$ yields the desired upper bound. 
\end{proof}

\subsection{Proof of Corollary \ref{p-UpperBoundGeneral}}
\begin{proof}[Proof of Corollary \ref{p-UpperBoundGeneral}]
Let $r$ be a defining function for $M$, and recall that $\bar g$ denotes the compactified metric $r^2g_+$, while $g$ denotes the complete metric $g_+$. Using the transformation law \eqref{TransformationRiemannTensor} for $\kappa^2g_+ = r^{-2}(\kappa^2\bar g)$, we obtain 
\begin{equation}
\begin{split}
 R_{ijkl}^{\kappa^2g_+}&=-|dr|^2_{\kappa^2\overline{g}}((\kappa^2g)_{ik}(\kappa^2g)_{jl}-(\kappa^2g)_{il}(\kappa^2g)_{jk})+O_{ijkl}(r^{-3})\\ &= -(\kappa^{-2}|d r|^2_{\overline{g}})((\kappa^2g)_{ik}(\kappa^2g)_{jl}-(\kappa^2g)_{il}(\kappa^2g)_{jk})+O_{ijkl}(r^{-3})
 \end{split}
 \end{equation}
By assumption the manifold $(X^{n+1},g_+)$ is AH$(-\kappa^2)$. As explained in the introduction, this means that $|dr|^2_{\bar g}|_{M} = \kappa^2 \iff \kappa^{-2}|d r|^2_{\bar g}|_{M} = 1$. This allow us to conclude that the sectional curvatures of $(X^{n+1},\kappa^2g_+)$ approach $-1$ at $M$, that is, $(X^{n+1},\kappa^2g_+)$ is asymptotically hyperbolic. Therefore, it follows from Theorem \ref{p-UpperBound} that $\lambda_{1,p}(X,\kappa^2g_+)\le \left(\frac{n}{p}\right)^p$. Finally, by the variational characterization $(\ref{p-Eigenvalue})$ of the first $p$-Dirichlet eigenvalue, the scaling property
\[
    \lambda_{1,p}(X,\kappa^2g_+) = \inf_\Omega \lambda_{1,p}(\Omega, \kappa^2g_+) = \kappa^{-p}\inf_\Omega \lambda_{1,p}(\Omega,g_+)\\ = \kappa^{-p}\lambda_{1,p}(X,g_+)
\]
holds and the result follows.\end{proof}

\section{Estimates on Submanifolds - Proofs}\label{SubmanifoldsOfAH}

\subsection{Proof of Proposition\ref{ACM}}\label{ACM-Sec}

We start with some preliminaries. Let $(x^{\alpha})$ be a coordinate system on a neighborhood of $Y$ containing a neghbourhood of the boundary of $Y$ within $Y.$ Let $(\mu_{\alpha'})$ be an orthonormal frame for $NY$ with respect to $g_+$ that is adapted in the following sense: if $g(\nabla_{\overline{g}}r,\mu_{\alpha'})<0$, then replace $\mu_{\alpha'}$ with $-\mu_{\alpha'}$. Let $\overline{\mu}_{\alpha'}=r^{-1}\mu_{\alpha'};$ this is a unit-normal vector for $\overline{g}$. Let $W$ and $Z$ be tangent vectors to $Y$ at a point and $\overline{W}$ and $\overline{Z}$ be any extensions to vector fields to $X$ which restrict at points of $Y$ to vector fields tangent to $Y.$ Then the vector-valued second fundamental form is given by
\begin{equation}
II(W,Z)=B^{\alpha'}(W,Z)\mu_{\alpha'}=\sum_{\alpha'}B_{\alpha'}(W,Z)\mu_{\alpha'},
\end{equation}
where 
\begin{equation}
B_{\alpha'}(W,Z)=g_+(\nabla^{g_+}_{\overline{W}}\overline{Z},\mu_{\alpha'}).
\end{equation}
For $q\in Y$, let $P_q\colon T_qX\to T_qY^{\perp}$ be the canonical normal-projection map. More explicitly, with respect to our chosen frame, this canonical projection takes the form $[v\mapsto \sum_{\alpha'}g(v,\mu_{\alpha'})\mu_{\alpha'}].$ We point out that this projection is independent of our choice of frame. Notice that $\sum_{\alpha'}g(v,\mu_{\alpha'})\mu_{\alpha'}=\sum_{\alpha'}\overline{g}(v,\overline{\mu}_{\alpha'})\overline{\mu}_{\alpha'}$, and observe that
\begin{equation}\label{Proj}
\sum_{\alpha'} \overline{g}(\overline{\mu}_{\alpha'},\nabla_{\overline{g}}r)^2=|P(\nabla_{\overline{g}}r)|^2_{\overline{g}}\leq |\nabla_{\overline{g}}r|^2_{\overline{g}}=1.
\end{equation}

We define $\Theta_{\alpha'}$ to be the angle between $\nabla_{\overline{g}}r$ and $\mu_{\alpha'}$:
\begin{equation}
\cos \Theta_{\alpha'}:=\overline{\mu}_{\alpha'}(r)=\overline{g}(\overline{\mu}_{\alpha'},\nabla_{\overline{g}}r)=g(\mu_{\alpha'},r\nabla_{\overline{g}}r).
\end{equation}
Note that $\Theta_{\alpha'}$ is manifestly a conformal invariant. Using this, we can now rewrite $P(\nabla_{\bar g}r)$ as 
\begin{equation}
P(\nabla_{\overline{g}}r)=\sum_{\alpha'}\cos \Theta_{\alpha'}\overline{\mu}_{\alpha'}.
\end{equation} 
This gives us an interpretation of $P(\nabla_{\overline{g}}r)$ as being a vector field which contains all the information about the angles formed between $\nabla_{\overline{g}}r$ and the orthonormal frame $\overline{\mu}_{\alpha'}.$ Note that if we choose a different orthonormal frame then the individual angles $\Theta_{\alpha'}$ may change but $P(\nabla_{\overline{g}}r)$ is invariant.\par
Recall the conformal transformation law 
\begin{equation}\label{2ff-TransformationLaw}
B_{\alpha\beta\alpha'}=\frac{\overline{B}_{\alpha\beta\alpha'}}{r}+\frac{\overline{\mu}_{\alpha'}(r)\overline{h}_{\alpha\beta}}{r^2}.
\end{equation}
If we impose our assumption, $g_+(H,H)=C^2+O(r)$, then 
\begin{equation}
H=\sum_{\alpha'}H_{\alpha'}\mu_{\alpha'}
\end{equation}
with
\begin{align}
H_{\alpha'}&=h^{\alpha\beta}B_{\alpha\beta\alpha'}\\
&=r\overline{h}^{\alpha\beta}\overline{B}_{\alpha\beta\alpha'}+\overline{\mu}_{\alpha'}(r)(k+1).
\end{align}
Therefore,
\begin{equation}\label{cos-angle}
\cos \Theta_{\alpha'}=\overline{\mu}_{\alpha'}(r)=\frac{H_{\alpha'}}{k+1}+O(r),
\end{equation}
thus it follows that 
\begin{align}\label{C}
|P(\nabla_{\overline{g}}r)|^2_{\overline{g}}=\frac{C^2}{(k+1)^2}+O(r).
\end{align}
After recalling (\ref{Proj}) and taking $r\to 0$, we conclude
\begin{equation}
C^2\leq (k+1)^2.
\end{equation}

Before we start with the proof of Proposition \ref{ACM}, we explain some interesting geometric consequences of the previous formulas. First, (\ref{C}) gives us that the multi-angle vector $P(\nabla_{\overline{g}}r)$ has constant length at infinity given by $\dfrac{C}{(k+1)}$. If we use (\ref{cos-angle}) and think in terms of our frame $\overline{\mu}_{\alpha'}$, then we deduce that the angle between $\overline{\mu}_{\alpha'}$ and $\nabla_{\overline{g}}r$ at infinity is given by $\arccos\left(\frac{H_{\alpha'}}{(k+1)}\right)$. Given $q\in\partial\overline{Y}$, we say that $\Theta:=\arccos(|P(\nabla r)(q)|)$ is the \textbf{generalized non-oblique angle} at which $Y$ meets $M$ at $q.$ In the hypersurface case, this gives us the non-oblique angle between $Y$ and $M$ at $q$ in the plane determined by any choice of normal vector to $Y$ and $\nabla_{\overline{g}}r.$ This can be thought of as a generalized angle.

 In particular, asymptotically minimal submanifolds ($C=0$) meet $M$ at right angles in the sense that all $\Theta_{\alpha'}=\dfrac{\pi}{2}$ no matter which normal frame we choose, and that asymptotically CMC submanifolds meet $M$ at a constant generalized angle. Furthermore, this gives an upper bound on the possible value of the asymptotic mean curvature for asymptotically CMC submanifolds; we have seen $C^2\leq (k+1)^2$. Note that CMC conformally compact submanifolds are special cases of these, so this bound applies to them globally. This shows that the bound is determined purely by the asymptotics. If $C=k+1$, then $P(\nabla_{\overline{g}}r)|_{r=0}=\nabla_{\overline{g}}r|_{r=0}$, since $\nabla_{\overline{g}}r|_{r=0}$ is normal to $M$ at $r=0$, and, therefore, it follows that $T_qY\subset T_qM$ for all $q\in \partial{\overline{Y}}$. This violates the assumption that $Y$ meets $M$ transversely, hence it must be true that $C<(k+1).$

Now we proceed with the proof.

\begin{proof}[Proof of Proposition \ref{ACM}]
We use the Gauss Equation, together with the transformation law given in (\ref{2ff-TransformationLaw}), to obtain
\begin{align*}
R^X_{\alpha\beta\gamma\delta}&=R^Y_{\alpha\beta\gamma\delta}+(B_{\alpha\delta\alpha'}B_{\beta\gamma\alpha'}-B_{\alpha\gamma\alpha'}B_{\beta\delta\alpha'})\\
&=R^Y_{\alpha\beta\gamma\delta}+\sum_{\alpha'}\Bigg{(}\left[\frac{\overline{B}_{\alpha\delta\alpha'}}{r}+\frac{\overline{\mu}_{\alpha'}(r)\overline{h}_{\alpha\delta}}{r^2}\right]\left[\frac{\overline{B}_{\beta\gamma\alpha'}}{r}+\frac{\overline{\mu}_{\alpha'}(r)\overline{h}_{\beta\gamma}}{r^2}\right]\\ &\hspace{1in}-\left[\frac{\overline{B}_{\alpha\gamma\alpha'}}{r}+\frac{\overline{\mu}_{\alpha'}(r)\overline{h}_{\alpha\gamma}}{r^2}\right]\left[\frac{\overline{B}_{\beta\delta\alpha'}}{r}+\frac{\overline{\mu}_{\alpha'}(r)\overline{h}_{\beta\delta}}{r^2}\right]\Bigg{)}\\
&=R^Y_{\alpha\beta\gamma\delta}+\frac{|P(\nabla_{\overline{g}}r)|^2}{r^4}(\overline{h}_{\alpha\delta}\overline{h}_{\beta\gamma}-\overline{h}_{\alpha\gamma}\overline{h}_{\beta\delta})+O_{\alpha\beta\gamma\delta}(r^{-3})\\
&=R^Y_{\alpha\beta\gamma\delta}+|P(\nabla_{\overline{g}}r)|^2(h_{\alpha\delta}h_{\beta\gamma}-h_{\alpha\gamma}h_{\beta\delta})+O_{\alpha\beta\gamma\delta}(r^{-3}).
\end{align*}
It now follows from \eqref{TransformationRiemannTensor} and \eqref{C} that 
\begin{align*}
R^Y_{\alpha\beta\gamma\delta}=
&\left[\frac{C^2}{(k+1)^2}-1\right](h_{\alpha\gamma}h_{\beta\delta}-h_{\alpha\delta}h_{\beta\gamma})+O_{\alpha\beta\gamma\delta}(r^{-3})\\
&=-\left[1-\frac{C^2}{(k+1)^2}\right](h_{\alpha\gamma}h_{\beta\delta}-h_{\alpha\delta}h_{\beta\gamma})+O_{\alpha\beta\gamma\delta}(r^{-3}).
\end{align*}
It follows that 
\begin{equation}
sec^Y(X,Z)=-\bigg[1-\frac{C^2}{(k+1)^2}\bigg]+O(r),
\end{equation}
where $(X,Z)$ is any pair of vector fields on $Y$ that always span a plane and are defined in some collar neighborhood. Since $Y$ is asymptotically CMC and we know by the discussion prior to the proof that $C<k+1,$ it now follows that $Y$ is $\mathrm{AH}\left(-\left(1-\frac{C^2}{(k+1)^2}\right)\right).$
\end{proof}

\subsection{Proof of Theorem \ref{lower}}\label{LowerBound}

In \cite{MazzeoRafe1991UCaI}, Mazzeo showed that the essential spectrum of an $(n+1)$-dimensional asymptotically hyperbolic manifold $(X,g_+)$ is $[\frac{n^2}{4},\infty)$ with no embedded eigenvalues, leaving the possibility of the existence of finitely many eigenvalues in $(0,\frac{n^2}{4})$. In particular, the first Dirichlet eigenvalue could be less than $\frac{n^2}{4}$. However, a few years after, Lee proved in \cite{LeeJohnM.1995Tsoa} that if the manifold is, in addition, Poincar\'e-Einstein and its conformal infinity has nonnegative Yamabe invariant, then the first Dirichlet eigenvalue is exactly $\frac{n^2}{4}$. As a consequence, such manifolds have no $L^2$ - eigenvalues. 

Since it will be relevant to our work, in what follows we explain Lee's approach in detail. Here is where the assumption that $g_+$ is conformally compact of order $C^{3,\alpha}$ becomes necessary. Thanks to a result of Barta (see Chapter 3 in \cite{ChavelIsaac1984EiRg}), the first Dirichlet eigenvalue $\lambda_{1,2}(\Omega)$ on any bounded domain with smooth boundary satisfies
\begin{equation}
\inf_{x\in \Omega} \left(\frac{-\Delta_{g_+} f}{f}\right)\le \lambda_{1,2}(\Omega),
\end{equation}
where $f$ is a smooth positive function vanishing on $\partial \Omega$. In \cite{ChengS.Y.1975Deor} (see Section 3), the authors extended Barta's result to complete, non-compact manifolds $(X,g_+)$ and showed that 
\begin{equation}\label{Barta2}
\inf_{x\in X} \left(\frac{-\Delta_{g_+} f}{f}\right)\le \lambda_{1,2}(X),
\end{equation}
where $f$ is a smooth positive function on $X$. 

With this result at hand, Lee's idea is to construct a positive smooth test function $\varphi$ on $X$ such that $\frac{-\Delta_{g_+}\varphi}{\varphi}\ge \frac{n^2}{4}$. The test function is constructed as follows. On any asymptotically hyperbolic manifold $X$ and for any smooth defining function $r$, there exists a unique, smooth and strictly positive function $u$ on $X$ satisfying
\begin{equation}\label{Lee-Eq}
\begin{cases}
    \Delta_{g_+} u &= (n+1)u\\
    u-r^{-1}&=O(1);
\end{cases}
\end{equation}
see Proposition 4.1 in \cite{LeeJohnM.1995Tsoa} - the sign discrepancy is due to different conventions in the definition of the Laplace operator.  Notice that if $s>0$ and we take $\varphi = u^{-s}$, then 
     \begin{equation}
         \frac{-\Delta_{g_+} \varphi}{\varphi} = -\text{div}_{g_+}(-su^{-s-1}\nabla_{g_+} u) =s(n+1) -s(s+1) \frac{|\nabla_{g_+}u|^2}{u^2}.
     \end{equation}
     Therefore, if $\frac{|\nabla_{g_+}u|^2}{u^2}\le 1$, we would then obtain \begin{equation}
         \frac{-\Delta_{g_+}\varphi}{\varphi} \ge s(n+1) - s(s+1) = s(n-s).
     \end{equation}
     Noticing that $s(n-s)$, as a function of $s>0$, has a maximum at $s = \frac{n}{2}$ would imply the result. 
     
     The main difficulty is proving that \begin{equation}\label{KeyEstimate}
         \frac{|\nabla_{g_+}u|^2}{u^2}\le 1
     \end{equation}
     holds globally on $X$. It turns out that if $(X^{n+1},g_+)$ is Poincar\'e-Einstein and the conformal infinity is of non-negative Yamabe type, then we can select a smooth defining function $r$ such that $(r^2g_+)|_{TM}$ has non-negative scalar curvature. For the corresponding solution $u$, (\ref{KeyEstimate}) holds and the result follows; see Theorem A in \cite{LeeJohnM.1995Tsoa}. This motivates our following definition:

     \begin{definition}[Lee-eigenfunction]\label{Lee-Eigenfunction} Let $(X^{n+1},g_+)$ be a PE space whose conformal infinity $(M,[g_+]_\infty)$ is of non-negative Yamabe type. As explained, for a smooth defining function $r$, there exists a unique, smooth positive solution $u$ solving (\ref{Lee-Eq}). If $(r^2g_+)|_{TM}$ has non-negative scalar curvature, then we say that $u$ is a Lee-eigenfunction.
     \end{definition}

     \begin{remark}
         As pointed out by Guillarmou-Qing in \cite{GuillarmouColin2010SCoP} (see page 5), (\ref{Lee-Eq}) and (\ref{KeyEstimate}) still remain to be true if the Einstein condition is replace by $\mathrm{Ric}_{g_+}\ge -ng_+$. This means that our Theorem \ref{LowerBound}, and its consequences, can be strengthened by replacing the Poincar\'e-Einstein with such lower bound on the ricci curvature. 
     \end{remark}

     Our submanifolds are not necessarily Poincar\'e-Einstein, and we do not impose any assumptions on its conformal boundary. However, they live inside a Poincar\'e-Einstein manifold whose conformal infinity has non-negative Yamabe invariant. That means that a Lee-eigenfunction exists on the ambient manifold, and it globally satisfies (\ref{KeyEstimate}). We utilize Barta's inequality (\ref{Barta2}) and Lee's trick applied to the restriction of $u^{-s}$ to $Y$ to derive the lower bound.

\begin{proof}[Proof of Theorem \ref{lower}]
Let $u$ be a Lee-eigenfunction on $(X,g_+)$, and set $\hat u:= u|_Y$. We denote by $h_+$ the induced metric on $Y$, and we use $g^{\alpha'\beta'}$, $\alpha',\beta' = k+2,\cdots, n+1$, to denote the components of $g_+^{-1}$ with respect to a local orthonormal frame for $TY^\perp$. Recall that (see Lemma 2 in \cite{ChoeJaigyoung1992Iiom}, for instance) 
\begin{equation}
\Delta_{h_+}\hat{u}=(\Delta_{g_+}u)|_{Y}+H^{\alpha'}u_{\alpha'}-(g^{\alpha'\beta'}\nabla^2_{\alpha'\beta'} u)|_{Y},
\end{equation}
and that, for a Lee-eigenfunction $u$, its trace-free hessian (w.r.t. $g_+$) equals $b(u) = \nabla^2_{g_+}u - ug_+$. Setting $T:=\text{tr}_{g_+}\left(b(u)|_{(TY^{\perp})^2}\right)$, we can therefore write
\begin{equation}
\begin{split}
\Delta_{h_+}\hat u &=(\Delta_{g_+}u)|_{Y}+H^{\alpha'}u_{\alpha'}-T-(n-k)\hat u
\\ &=(k+1)\hat u+H^{\alpha'}u_{\alpha'}-T.
\end{split}
\end{equation}
Next, for $s>0$, consider the test function $\varphi := \hat{u}^{-s}$,
and compute 
\[
\begin{split}
\frac{-\Delta_{h_+} \varphi}{\varphi} &= s\frac{\Delta_{h_+} \hat{u}}{\hat{u}}-s(s+1)\frac{|d\hat{u}|^2_{h_+}}{\hat{u}^2}\\&=s\left(k+1+H^{\alpha'}\log(u)_{\alpha'}-\frac{T}{\hat u}\right) -s(s+1) \frac{|\nabla_{h_+} \hat{u}|^2}{\hat{u}^2}.
\end{split}
\]
Using that $\frac{|\nabla_{h_+}\hat{u}|^2}{\hat u^2}\leq \frac{|\nabla_{g_+}u|^2}{ u^2}\leq 1$ on $Y$, we derive 
\[
\begin{split}
         \frac{-\Delta_{h_+}\varphi}{\varphi} &\ge s(k+1) - s(s+1)+ sH_{\alpha'} \log(u)_{\alpha'}-s\frac{T}{\hat u}\\&= s(k-s)+sH^{\alpha'}\log(u)_{\alpha'}-s\frac{T}{\hat u}.
\end{split}
\]
Now the Cauchy-Schwarz inequality gives 
\[
H^{\alpha'}\log(u)_{\alpha'}=H^{\alpha'}\frac{u_{\alpha'}}{\hat u}\geq -\alpha \frac{|\nabla_{g_+}u|}{ u}\geq -\alpha.
\] 
Therefore,
 \begin{equation}
         \frac{-\Delta_{h_+}\varphi}{\varphi} \ge s(k-s-\alpha-\beta^Y(u)).
         \end{equation}
Observe that $s(k-s-\alpha-\beta^Y(u))$, as a function of $s>0$, has a maximum at $s = \frac{k-\alpha-\beta^Y(u)}{2}$. The result now follows from (\ref{Barta2}).
\end{proof}

\subsection{Proof of Corollary \ref{cor1}}

\begin{proof}[Proof of Corollary \ref{cor1}]
    Notice that for any Lee-eigenfunction $u$, that is, for any $u\in \mathcal{L}(g_+)$, inequality (\ref{lower-estimate1}) holds. The result follows after taking the supremum over $\mathcal{L}(g_+)$ on both sides.
\end{proof}

\subsection{Proof of Corollary \ref{cor2}}

\begin{proof}[Proof of Corollary \ref{cor2}]
Without loss of generality assume $0\notin Y.$ Let $p\in \mathbb{H}^{n+1}\setminus Y.$ Define $x\colon \mathbb{H}^{n+1}\to [0,\infty)$ by $x(y):=d_{g_H}(p,y)$ where $g_H$ is the hyperbolic metric. Then (Lemma 3 in \cite{ChoeJaigyoung1992Iiom})
\begin{equation}\label{b=0}
\nabla ^2_{g_H} \cosh(x(y))=\cosh(x(y))g_H,
\end{equation}
and so $\Delta_{g_H} \cosh(x(y)) = (n+1)\cosh(x(y))$. Let $r(y)$ be a smooth defining function for $\mathbb{S}^n$ that agrees with $1-|y|$ outside of some compact set containing $0.$ Then $\cosh(x(y)) - r^{-1}(y)$ remains bounded as $|y|\to 1$, therefore $u(y)=\cosh(x(y))$ is a Lee-eigenfunction for which $\beta^Y(u) = 0$ thanks to \eqref{b=0}. 
\end{proof}

\subsection{Stability - Proof of Corollary \ref{Stability}}
\begin{proof}[Proof of Corollary \ref{Stability}]
Let $f$ be a compactly supported Lipschitz function, and recall from Corollary \ref{ImprovedLower} in the minimal case ($\alpha = 0$) that
\begin{equation}
\frac{(n-1-\hat{\beta}^Y)^2}{4}\leq\lambda_{1,2}(Y)\leq\frac{\displaystyle\int_Y|\nabla f|^2\;dv_{h_+}}{\displaystyle\int_Yf^2\;dv_{h_+}}.
\end{equation}
Therefore,
\[
\begin{split}
\int_Y(|\nabla f|^2-(|B|^2-n)f^2)\;dv_{h_+}&\geq \int_Y\left(\frac{(n-1-\hat{\beta}^Y)^2}{4}+n - |B|^2\right)f^2\;dv_{h_+}\\&\ge 0,
\end{split}
\]
and the result follows. The second part of the statement follows from the fact that when the ambient manifold is $\mathbb{H}^{n+1}$, then $\hat\beta^Y =0$ for any such submanifold; see proof Corollary \ref{cor2}.
\end{proof}

\begin{remark}\label{Seo}
    In Seo's paper \cite{SeoKeom-Kyo2011SMHI}, the work of Cheung-Leung is quoted incorrectly. For instance, Cheung-Leung's lower bound estimate does not require stability; see Theorem 2 in \cite{CheungLeung-Fu2001Eefs}. Also, in the statement of Theorem 3.1 in \cite{SeoKeom-Kyo2011SMHI}, the norm of the second fundamental form should be squared.
\end{remark}

\subsection{Proof of Corollary \ref{GeomInt}}

\begin{proof}
    Recall that $\mathrm{Ric}(g_+) = -ng_+$, $b(u) = \nabla^2_{g_+}u - ug_+$ and $g_u = u^{-2}g_+$. Using the conformal transformation law for the Ricci tensor, we obtain
    \begin{equation}
    \begin{split}
    \text{Ric}_{g_u} &= g_+ + (n-1)u^{-1} \nabla^2_{g_+}u - nu^{-2}|\nabla_{g_+}u|^2g_+ \\ &= ng_+ + (n-1)u^{-1}b(u) - n|\nabla_{g_+}u|^2g_u.
    \end{split}
    \end{equation}
Since $\text{tr}_{g_u}(\text{Ric}_{g_u}) = R_{g_u} = n(n+1)(u^2 - |\nabla^2_{g_+}u|^2)$ and $E_{g_u} = \text{Ric}_{g_u} - \frac{R_{g_u}}{n+1}g_u$, we get
\begin{equation}
\begin{split}
    E_{g_u} &= ng_+ +(n-1)u^{-1}b(u) - n|\nabla_{g_+}u|^2g_u - ng_+ + n|\nabla_{g_+}u|^2g_u \\&= (n-1)u^{-1}b(u),
\end{split}
\end{equation}
as desired. 

For the second part of the statement, notice that now we can conclude that $b(u) = 0$ if and only if $E_{g_u}=0$, thus either condition implies that $R_{g_u}$ is constant. On the other hand, if $R_{g_u}$ is constant, then it is harmonic and using Bochner's formula we conclude that $|b(u)|^2_{g_+} = 0$. This concludes the proof.
\end{proof}

\section{Appendix: One-to-one correspondence between \texorpdfstring{$\mathcal{L}(g_+)$}{} and \texorpdfstring{$[g_+]_\infty^{Sc\geq 0}$}{}}
We refer to section \ref{LowerBound} for some of the terminology employed here.

\begin{proposition}\label{LEE}
    Let $\mathcal{L}(g_+)$ be the set of all Lee-eigenfunctions on a PE space $(X^{n+1},g_+)$ and let $[g_+]_{\infty}^{Sc\geq 0}$ be the subset of $[g_+]_{\infty}$ whose elements have non-negative scalar curvature. Then there is a $1-1$ correspondence between $\mathcal{L}(g_+)$ and $[g_+]_{\infty}^{Sc\geq 0}.$
\end{proposition}
\begin{proof}
Define the map $\Omega\colon [g_+]_{\infty}^{Sc\geq 0}\to\mathcal{L}(g_+)$ as follows. 
For $\hat{g}\in [g_+]_{\infty}^{Sc\geq 0}$, let $r$ be a defining function such that $(r^2g_+)|_{TM} = \hat{g}$. Set $\Omega(\hat{g}):=u$, where $u$ is the Lee-eigenfunction determined by $r$. We claim that this is independent of the choice of $r$ consider. Indeed, let $\rho$ be another defining function such that $(\rho^2g_+)|_{TM}=\hat{g}=(r^2g_+)|_{TM}$. Then $\rho=r+O(r^2)$, and so $u-r^{-1}=u-(\rho+O(r^2))^{-1}=u-\rho^{-1}+O(\rho^2).$ Since $u-r^{-1}$ is bounded as $r\to 0$, we conclude that $u-\rho^{-1}$ is also bounded as $\rho\to 0$. It follows by uniqueness that $u$ is the eigenfunction which is induced by $\rho$, showing the claim. 

The map $\Omega$ is surjective by construction. We proceed to showing that $\Omega$ is injective. To this end, suppose $\Omega(\hat{g})=\Omega(\hat{g}')=u$, and  let $r$ and $r_0$ be defining functions which correspond to $\hat{g}$ and $\hat{g}'$, respectively. Then $u-r^{-1}$ and $u-r_0^{-1}$ are both bounded as $r,r_0\to 0.$ Consequently, $(u-r^{-1})-(u-r_0^{-1})$ is bounded as $r\to 0$, which implies
$r^{-1}-r_0^{-1}$ is bounded too as $r\to 0$. We know there exists $c\in\mathbb{R}$ such that $r=cr_0+O(r_0^2).$ It follows that 
$c^{-1}r_0^{-1}-r_0^{-1}$ is bounded as $r\to 0$, therefore 
$(c^{-1}-1)r_0^{-1}$ is bounded as $r\to 0$.
It follows that $c=1$ and therefore \begin{align}(r^2g_+)|_{TM}&=([r_0+O(r_0^2)]^2g_+)|_{TM}\\\nonumber
&=(r_0^2g_+)|_{TM}.
\end{align}

\end{proof}

\bibliographystyle{amsplain}
\bibliography{bibliography.bib}

\begin{bibdiv}
\begin{biblist}

\bib{BrooksRobert1981Arbg}{article}{
      author={Brooks, Robert},
       title={A relation between growth and the spectrum of the laplacian},
    language={eng},
        date={1981},
        ISSN={0025-5874},
     journal={Mathematische Zeitschrift},
      volume={178},
      number={4},
       pages={501\ndash 508},
}

\bib{DoCarmo-Dj}{article}{
      author={Carmo, Manfredo P.~Do},
      author={Dajczer, M.},
       title={Rotation hypersurfaces in spaces of constant curvature},
    language={eng},
        date={2012},
     journal={Springer},
}

\bib{Case}{article}{
      author={Case, Jeffrey~S.},
      author={Graham, C.~Robin},
      author={Kuo, Tzu-Mo},
      author={Tyrrell, Aaron~J.},
      author={Waldron, Andrew},
       title={A gauss-bonnet formula for the renormalized area of minimal submanifolds of poincar\'e-einstein manifolds},
        date={2024},
     journal={arXiv preprint arXiv:2403.16710},
}

\bib{ChavelIsaac1984EiRg}{book}{
      author={Chavel, Isaac},
       title={Eigenvalues in riemannian geometry},
    language={eng},
      series={Pure and applied mathematics (Academic Press) ; 115},
   publisher={Academic Press},
     address={Orlando},
        date={1984},
        ISBN={0121706400},
}

\bib{ChengS.Y.1975Deor}{article}{
      author={Cheng, S.~Y.},
      author={Yau, S.~T.},
       title={Differential equations on riemannian manifolds and their geometric applications},
    language={eng},
        date={1975},
        ISSN={0010-3640},
     journal={Communications on pure and applied mathematics},
      volume={28},
      number={3},
       pages={333\ndash 354},
}

\bib{ChengShiu-Yuen1975Ecta}{article}{
      author={Cheng, Shiu-Yuen},
       title={Eigenvalue comparison theorems and its geometric applications},
    language={eng},
        date={1975},
        ISSN={0025-5874},
     journal={Mathematische Zeitschrift},
      volume={143},
      number={3},
       pages={289\ndash 297},
}

\bib{CheungLeung-Fu2001Eefs}{article}{
      author={Cheung, Leung-Fu},
      author={Leung, Pui-Fai},
       title={Eigenvalue estimates for submanifolds with bounded mean curvature in the hyperbolic space},
    language={eng},
        date={2001},
        ISSN={0025-5874},
     journal={Mathematische Zeitschrift},
      volume={236},
      number={3},
       pages={525\ndash 530},
}

\bib{ChoeJaigyoung1992Iiom}{article}{
      author={Choe, Jaigyoung},
      author={Gulliver, Robert},
       title={Isoperimetric inequalities on minimal submanifolds of space forms},
    language={eng},
        date={1992},
        ISSN={0025-2611},
     journal={Manuscripta mathematica},
      volume={77},
      number={1},
       pages={169\ndash 189},
}

\bib{coskunuzer2009asymptotic}{article}{
      author={Coskunuzer, Baris},
       title={Asymptotic plateau problem},
        date={2009},
     journal={arXiv preprint arXiv:0907.0552},
}

\bib{CarvalhoFranciscoG.2022Otft}{article}{
      author={de~S.~Carvalho, Francisco~G.},
      author={de~A.~Cavalcante, Marcos~P.},
       title={On the fundamental tone of the p-laplacian on riemannian manifolds and applications},
    language={eng},
        date={2022},
        ISSN={0022-247X},
     journal={Journal of mathematical analysis and applications},
      volume={506},
      number={2},
       pages={125703},
}

\bib{DuMao}{article}{
      author={Du, Feng},
      author={Mao, Jing},
       title={Estimates for the first eigenvalue of the drifting laplace and the p-laplace operators on submanifolds with bounded mean curvature in the hyperbolic space},
    language={eng},
        date={2017},
        ISSN={0022-247X},
     journal={Journal of mathematical analysis and applications},
      volume={456},
      number={2},
       pages={787\ndash 795},
}

\bib{FuTao}{article}{
      author={Fu, Haiping},
      author={Tao, Yongqian},
       title={Eigenvalue estimates for complete submanifolds in the hyperbolic spaces},
    language={eng},
        date={2013},
     journal={Journal of Mathematical Research with Applications},
      volume={33},
      number={5},
       pages={598\ndash 606},
}

\bib{Graham-Witten}{article}{
      author={Graham, C.~R.},
      author={Witten, Edward},
       title={Conformal anomaly of submanifold observables in ads/cft correspondence},
    language={eng},
        date={1999},
        ISSN={0550-3213},
     journal={Nuclear physics. B},
      volume={546},
      number={1},
       pages={52\ndash 64},
}

\bib{Graham}{article}{
      author={Graham, C~Robin},
       title={Volume and area renormalizations for conformally compact einstein metrics},
    language={eng},
        date={1999},
        ISSN={2331-8422},
     journal={arXiv.org},
}

\bib{Gursky-Graham}{misc}{
      author={Graham, C.~Robin},
      author={Gursky, Matthew~J.},
       title={Chern-gauss-bonnet formula for singular yamabe metrics in dimension four},
        date={2019},
         url={https://arxiv.org/abs/1902.01562},
}

\bib{GrahamLee}{article}{
      author={Graham, C.Robin},
      author={Lee, John~M},
       title={Einstein metrics with prescribed conformal infinity on the ball},
    language={eng},
        date={1991},
        ISSN={0001-8708},
     journal={Advances in mathematics (New York. 1965)},
      volume={87},
      number={2},
       pages={186\ndash 225},
}

\bib{GuillarmouColin2010SCoP}{article}{
      author={Guillarmou, Colin},
      author={Qing, Jie},
       title={Spectral characterization of poincaré–einstein manifolds with infinity of positive yamabe type},
    language={eng},
        date={2010},
        ISSN={1073-7928},
     journal={International mathematics research notices},
      volume={2010},
      number={9},
       pages={1720\ndash 1740},
}

\bib{HijaziOussama2020TCCo}{article}{
      author={Hijazi, Oussama},
      author={Montiel, Sebastián},
      author={Raulot, Simon},
       title={The cheeger constant of an asymptotically locally hyperbolic manifold and the yamabe type of its conformal infinity},
    language={eng},
        date={2020},
        ISSN={0010-3616},
     journal={Communications in mathematical physics},
      volume={374},
      number={2},
       pages={873\ndash 890},
}

\bib{LEAn2006Epft}{article}{
      author={LE, An},
       title={Eigenvalue problems for the p-laplacian},
    language={eng},
        date={2006},
        ISSN={0362-546X},
     journal={Nonlinear analysis},
      volume={64},
      number={5},
       pages={1057\ndash 1099},
}

\bib{LeeJohnM.1995Tsoa}{article}{
      author={Lee, John~M.},
       title={The spectrum of an asymptotically hyperbolic einstein manifold},
    language={eng},
        date={1995},
        ISSN={1019-8385},
     journal={Communications in analysis and geometry},
      volume={3},
      number={2},
       pages={253\ndash 271},
}

\bib{LiZhi2020UBot}{article}{
      author={Li, Zhi},
      author={Huang, Guangyue},
       title={Upper bounds on the first eigenvalue for the p-laplacian},
    language={eng},
        date={2020},
        ISSN={1660-5446},
     journal={Mediterranean journal of mathematics},
      volume={17},
      number={4},
}

\bib{Lima}{article}{
      author={Lima, Barnabé~Pessoa},
      author={Montenegro, José Fábio~Bezerra},
      author={Santos, Newton~Luís},
       title={Eigenvalue estimates for the p -laplace operator on manifolds},
    language={eng},
        date={2010},
        ISSN={0362-546X},
     journal={Nonlinear analysis},
      volume={72},
      number={2},
       pages={771\ndash 781},
}

\bib{LindqvistPeter}{article}{
      author={Lindqvist, Peter},
       title={On the equation $\text{div}(|\nabla u|^{p-2}\nabla u)+ \lambda|u|^{p-2}u=0$},
    language={eng},
        date={1990},
        ISSN={0002-9939},
     journal={Proceedings of the American Mathematical Society},
      volume={109},
      number={1},
       pages={157\ndash 164},
}

\bib{marx2021variations}{article}{
      author={Marx-Kuo, Jared},
       title={Variations of renormalized volume for minimal submanifolds of poincare-einstein manifolds},
        date={2021},
     journal={arXiv preprint arXiv:2111.04294},
}

\bib{MazzeoRafe1991UCaI}{article}{
      author={Mazzeo, Rafe},
       title={Unique continuation at infinity and embedded eigenvalues for asymptotically hyperbolic manifolds},
    language={eng},
        date={1991},
        ISSN={0002-9327},
     journal={American journal of mathematics},
      volume={113},
      number={1},
       pages={25\ndash 45},
}

\bib{mazzeo1986hodge}{thesis}{
      author={Mazzeo, Rafe~Roys},
       title={Hodge cohomology of negatively curved manifolds},
        type={Ph.D. Thesis},
        date={1986},
}

\bib{McKean}{article}{
      author={McKean, H.~P.},
       title={An upper bound to the spectrum of {$\Delta$} on a manifold of negative curvature},
    language={eng},
        date={1970},
        ISSN={0022-040X},
     journal={Journal of differential geometry},
      volume={4},
      number={3},
       pages={359\ndash 366},
}

\bib{QingJie2003Otrf}{article}{
      author={Qing, Jie},
       title={On the rigidity for conformally compact einstein manifolds},
    language={eng},
        date={2003},
        ISSN={1073-7928},
     journal={International Mathematics Research Notices},
      volume={2003},
      number={21},
       pages={1141\ndash 1153},
}

\bib{SeoKeom-Kyo2011SMHI}{article}{
      author={Seo, Keom-Kyo},
       title={Stable minimal hypersurfaces in the hyperbolic space},
    language={eng},
        date={2011},
        ISSN={0304-9914},
     journal={Journal of the Korean Mathematical Society},
      volume={48},
      number={2},
       pages={253\ndash 266},
}

\bib{SungChiung-JueAnna2014Sgea}{article}{
      author={Sung, Chiung-Jue~Anna},
      author={Wang, Jiaping},
       title={Sharp gradient estimate and spectral rigidity for $p$-laplacian},
    language={eng},
        date={2014},
        ISSN={1073-2780},
     journal={Mathematical research letters},
      volume={21},
      number={4},
       pages={885\ndash 904},
}

\end{biblist}
\end{bibdiv}

\end{document}